\documentclass[11pt]{article}
\usepackage[utf8]{inputenc}
\usepackage[T1]{fontenc}
\usepackage{amsmath, amssymb, amsthm}
\usepackage{xcolor}
\usepackage{geometry}
\usepackage{comment}

\usepackage{amsthm, amsfonts, amssymb}
\usepackage{mathrsfs}
\usepackage{amsmath}

\usepackage{natbib}

\newtheorem{theorem}{Theorem}[section]
\newtheorem{proposition}[theorem]{Proposition}
\newtheorem{lemma}[theorem]{Lemma}
\newtheorem{corollary}[theorem]{Corollary}
\newtheorem{remark}[theorem]{Remark}

\newcommand\R{\mathbb{R}}

\begin{document}
\title{Propagation of Regularity for Schr\"odinger Equations with Time Dependent Potentials}

\author{Avy Soffer \thanks{Rutgers University, Mathematics department, 110 Frelinghuysen Road, Piscataway, NJ 08854, US}}

\date{\today}
\maketitle

\begin{abstract}
The dynamics of Schr\"odinger equation with time dependent potentials of general time dependence is considered.
It is shown that for localized in space potentials, there is propagation of regularity which is uniformly bounded in higher Sobolev norms.
Unlike the cases where the solution scatter, and then propagation is proved via a standard bootstrap argument, the solutions considered here have a part that does not scatter, as expected in general. 
For this we introduce propagation estimates that work directly in (e.g.) $H^2(\R^3).$

             We prove Propagation of Regularity for Schr\"odinger Equations with Time Dependent Potentials, namely, the $H^2$ norm of the solution remains uniformly bounded in time. We also prove some general Propagation Estimates that may be of interest.
\end{abstract}

\section{Introduction}
The Schr\"odinger equation with time dependent potential is fundamental; it corresponds to open quantum systems, as well as nonlinear equations. It appears naturally in linearization of non-linear dynamics, in long range N-body scattering, see e.g. \cite{sigal1993asymptotic} and charge transfer hamiltonians \cite{yajima1980multi}. See also \cite{chen2026trapped}.

Yet, it is under-developed at the theory level, and for a good reason: one cannot use the same type of spectral theory (following linear algebra) as in the time independent case. Hence, the well known methods and results are based on special cases which are mostly perturbative, like small perturbations, adiabatic time dependence, time periodic or vanishing as time goes to infinity.

On the mathematical side, we have some well-known important results on the theory of global existence; see, e.g. \cite{yosida2012functional,reed1975ii,kitada1982scattering,kitada1983remarks,ruiz1994local,linares2014introduction}. In particular, it is shown that for localized potentials that we consider below, the solutions exist in $H^2.$ However, these estimates do not bound the size, since the $H^2$ can grow exponentially in time.

In many cases, the relevant dispersive estimates for the solution can be proved, such as the Strichartz estimate or $L^p$ decay estimates. This can then be used to prove the uniform boundedness of higher norms by applying Gronwall's inequalities to derivatives of the equation. This can be done for defocusing Nonlinear wave equations.

For problems with time dependent potentials, it can be done for special types of time dependence\cite{ rodnianski2004time,beceanu2019semilinear,soffer2025local,rodnianski2005dispersive,fishman2014multiscale,teufel2003adiabatic}. Although there are results on scattering theory for time dependent potentials, they do not have the needed dispersive estimates, \cite{kitada1983remarks,enss1983bound}.
For small potentials, the N-body case was treated by \cite{yajima2025boundedness}

In recent years scattering theory was developed to include both time dependent potentials and nonlinear equations with focusing part.
\cite{T2008,liu2020general,liu2025large,soffer2022large,soffer2025ℒ}. They rely on the assumption that the solution is uniformly bounded in $H^1.$ They do not prove decay estimates which are useful for propagation of regularity. However, it should be noted that in the works of Tao, and later in the works of Liu-Soffer and Soffer-Wu, the weakly localized part of the solution is shown to be smooth in some cases.

In contrast with these methods, in this work a direct approach is used to control the higher Sobolev norm, avoiding the proofs of decay estimates.
Instead, we prove Propagation Estimates(PRES) directly in higher norms. Such estimates are typically proven by the use of Propagation Observables (PROB).
PROB is an operator $B(t)$ (or a family of operators depending on time) which is self adjoint and such that its expectation on the state at time $t$, $\langle \psi(t), B(t)\psi(t)\rangle$ has a non-negative time derivative.
There are many such known operators which give useful  apriori estimates on solutions of both linear and non-linear equations.
The first class is composed of operators which correspond to symmetries of the equation, and lead to conservation laws. The constant function implies $L^2$ conservation. Functions of the hamiltonian $H$ in the time independent case (energy conservation) etc...
A second class of operators are derived from the dispersive nature of the solutions of defocusing equations, and include the Morawetz estimate and the Dilation identity.

Modern estimates are now based on microlocalized objects, where only the leading order term of the derivative is non-negative, and there are higher order corrections, which need to be controlled in $L^1(dt).$ We will follow tools from \cite{sigal1986asymptotic, graf1990asymptotic,derezinski2013scattering,HSS,sigal1987n,sigal1988local,liu2025large,soffer2022large,soffer2011monotonic}. See also \cite{soffer2024new} for a  review.

So the typical estimate looks like 
\[ \langle \psi(T), B(T)\psi(T)\rangle -\langle \psi(t), B(t)\psi(t)\rangle=\int_t^T \langle \psi(s), C^*C(s)\psi(s)\rangle ds+
\]
\[
+\int_t^T \langle \psi(s), R(s)\psi(s)\rangle ds. \quad \langle R(s)\rangle \in L^1(ds).
\]
Here we introduce and use the notation $\langle\psi(t), Q(t)\psi(t)\rangle\equiv  \langle Q(t)\rangle. $

The proof of the propagation estimates is done by proving on various parts of the phase-space the propagation of regularity using the above argument and generalizations thereof, including summing over time slices, and summing over dyadic (in spectral space) PROBs.

\section{Smoothness of the Solution on Compact Domains in Space}
We consider the Schr\"odinger Equation (SE)
$$
\begin{aligned}
& i \frac{\partial \psi}{\partial t}=-\Delta \psi+V(x, t) \psi\equiv (H_0+V)\psi, \\
& \psi(0) \in H^{s}(\R^n), \quad s= 2, \quad n\geq 3.
\end{aligned}
$$
$V(x, t)$ is assumed to be localized in $x$ and smooth:
$$
\begin{aligned}
& \left\|\langle x\rangle^{\sigma} V(x, t)\right\|_{L^{\infty}}< c, \quad \sigma > 4 . \\
& \left\|\langle x\rangle^{\sigma}\left[\left|\frac{\partial V}{\partial t}\right|+\left|\nabla_{x} V\right|+|\Delta_x V|\right]\right\|_{L^{\infty}}< c, \quad \sigma > 4 .
\end{aligned}
$$
Here, we use the standard definition 
$$
\langle x\rangle^2 \equiv 1+|x|^2.
$$
\begin{remark}
These conditions are not optimal, and apply for $s \leq 2$. In particular, the restrictive $L^{\infty}$ conditions are used to streamline the presentation.
\end{remark}
The main result of the paper is the following theorem:
\begin{theorem}[Propagation of Regularity]

Under the above conditions on $V$ and the initial data, we have uniform boundedness of the $H^2$ norm of the solution:
 \[
\sup _{t}\left\| \Delta \psi(t)\right\|_{L^{2}} \lesssim\|\psi(0)\|_{H^{2}} .
\]\label{Local1}
\end{theorem}

The first step of the proof is to show that the solution is uniformly bounded in $\mathrm{H}^{2}$ (say), locally in space.

\begin{theorem}[Local Smoothness]

 Let $\psi(t)$ be a solution of the (SE) with initial data in $\mathrm{H}^{2}(\R^n), n\geq 3.$

Then,
\[
\sup _{t}\left\|\langle x\rangle^{-2-\delta} \Delta \psi(t)\right\|_{L^{2}} \lesssim\|\psi(0)\|_{H^{2}} .
\]\label{Local1}
\end{theorem}


\begin{proof}
We use the Duhamel representation.
$$
\langle x\rangle^{-\sigma} \Delta \psi(t)=\langle x\rangle^{-\sigma} e^{-i H_{0} t} \Delta \psi(0)-i \int_{0}^{t}\langle x\rangle^{-\sigma} \Delta e^{-i H_{0}(t-s)} V\left(x, s\right) \psi(s) d s
$$

The first term is bounded in $H^{2}$, since $\psi(0) \in H^{2}$ by assumption.
The Duhamel term, we only need to control it in $L^{2}$ for $p\left(\equiv-i \nabla_{x}\right)$ large.
Split the integral of time into two parts $\int_{t-\varepsilon}^{t}$ and $\int_{0}^{t-\varepsilon}$.
The part $[t- \varepsilon, t]$ is controlled by Bootstrap. The rest is controlled by the smoothing property of the free flow.

\begin{equation}\label{Duhamel}
\begin{aligned}
& \langle x\rangle^{-\sigma} \int_{0}^{t-\varepsilon} \Delta e^{i \Delta(t-s)} V(x, s) \psi(s) d s \\
& =\int_{0}^{t-\varepsilon}\langle x\rangle^{-\sigma} \Delta\langle x-2 p\tau\rangle^{-2} e^{i \Delta(t-s)}\langle x\rangle^{2} V(x, s) \psi(s) d s \\
& \tau \equiv t-s
\end{aligned}
\end{equation}
 Next, we need
\begin{proposition}[Local Decay and Smoothing]

For $\sigma =2,$
\begin{equation}\label{LD1}
\left < x\right >^{-\sigma} \langle x-2 p \tau\rangle^{-2} \leq \mathcal{O}( \left < x\right >^2+4p^2\tau^2)^{-1}\equiv B\left < x\right >^{-2}+C p^2\tau^{-2}, 
\end{equation}
for some bounded operators $B,C.$
\end{proposition}
\begin{proof}

The proof follows an algebraic calculation: We first note that the domain of essential self adjointness, the Schwarz class function, is common to the operators $x^2, \Delta, x, p, A.$ Therefore all expressions below make sense on the dense set consisting of Schwarz class functions. This space is left invariant under the action of the above operators. It is also left invariant under the action of the free flow.
$$
\begin{gathered}
\langle x\rangle^{-2}(-\Delta)\langle x-2 p \tau\rangle^{-2}=\langle x\rangle^{-2}\left(-\Delta+x^{2} / 4 \tau^{2}-x^{2} / 4 \tau^{2}\right)\langle x-2 p \tau\rangle^{-2} \\
=\langle x\rangle^{-2}\left(-\Delta+x^{2} / 4 \tau^{2}+1\right)\left\langle-\Delta 4 \tau^{2}+x^{2}+1\right\rangle^{-1}\left[\left(1+x^{2}-4 \Delta \tau^{2}\right)\langle x-2 p \tau)^{-2}\right] \\
+\left(O\left(\langle x\rangle^{-2}\right)+O\left(\tau^{-2}\right)\right)\langle x-2 p \tau\rangle^{-2} 
\end{gathered}$$
$$
\begin{gathered}
\langle x\rangle^{-2}\left(1+x^{2}-4 \Delta \tau^{2}\right)\langle x-2 p \tau\rangle^{-2}=\langle x \rangle^{-2}\left(1+x^{2}-4 \tau^{2} \Delta\right)\left[\frac{-1}{1+x^{2}-4 \Delta \tau^ 2}+\frac{1}{\langle x-2 p\tau\rangle ^{2}}\right]+ \\
+\langle x\rangle^{-2}=\langle x\rangle^{-2} \\
+\langle x\rangle^{-2}\left(1+x^{2}-4 \tau^{2} \Delta\right) \frac{1}{1+x^{2}-4 \Delta \tau^2} 2 A \tau \frac{1}{\langle x-2 p \tau\rangle^{2}} \\
=\langle x\rangle^{-2}+\langle x\rangle^{-2} 2 A \tau \frac{1}{\langle x-2 p \tau\rangle^{2}}.\\
A \equiv \frac{1}{2} x \cdot p+\frac{1}{2} p \cdot x=x \cdot p+c
\end{gathered}
$$

To control the $A$ term, we use (for $\tau <1$)
$$
\begin{aligned}
2 A \tau & =2 \tau(x \cdot p+c)= \\
& =4 \tau^{2}\left(c / \tau+\left(\frac{x}{2 \tau} \frac{1}{\delta}\right) \cdot(\delta p)\right) \\
& \leqslant c \tau+ \tau^{2}\left(p^{2} \delta^{2}+\frac{1}{\delta^{2}} \frac{x^{2}}{4 \tau^{2}}\right) \\
& =c \tau+4 \tau^{2}(-\Delta) \delta^{2}+\frac{1}{\delta^{2}} x^{2} \\
& =c \tau+4 \tau^{2}(-\Delta) \delta^{2}+\delta^{2} x^{2}+\left(\frac{1}{\delta^{2}}-1\right) x^{2}
\end{aligned}
$$

Therefore,
$$
\begin{aligned}
& \langle x\rangle^{-2} 2 A \tau\langle x-2 p \tau\rangle^{-2}= \\
& =O(\tau\rangle\langle x\rangle^{-2}+\delta^{2}\langle x\rangle^{-2}\left(-4 \tau^{2} \Delta+x^{2}\right)\langle x-2 p \tau\rangle^{-2} \\
& +\left(\frac{1}{\delta^{2}}-1\right) O(1)\langle x-2 p \tau\rangle^{-2}
\end{aligned}
$$

By choosing $\delta$ sufficiently small, the $\delta^{2}$ term is absorbed by the LHS.
For $\tau>1$, we use the same estimate on
$$
\left(\frac{x}{2 \tau} \cdot \frac{1}{\delta}\right) \cdot(\delta p)+(\delta p) \cdot\left(\frac{x}{2 \tau} \frac{1}{\delta}\right) \text {, which }
$$
removes the term $O(\tau)\langle x\rangle^{-2}$.
\end{proof}
With the above proposition, we proceed to control the first time interval of \eqref{Duhamel}.

We need to show that the integral over $[0,+t-\varepsilon]$ is uniformly bounded in $L^{2}$.
In the control of the $A$ term, we lose a factor of $\langle x\rangle^{-2}$.
Therefore, to get the integral bound, we need to use an extra factor of $\langle x\rangle^{-1-\delta}$ :
$\left\|\langle x\rangle^{-1-\delta} e^{i \Delta(t-s)} F_{2}(p)\langle x\rangle^{-1-\delta}\right\|_{L^{2} \rightarrow L^{2}} \leq c\langle t-s\rangle^{-1-\delta}.$

The above estimates show that for $|p| \geqslant 1$ the localized solution is uniformly bounded in $H^{2}$.
It remains to control the integral part of $[0, t-\varepsilon]$ with low frequency and control
the Duhamel term on $[t-\varepsilon, t]$.

{\bf Low Frequency}

For low frequency, we use the standard estimate $L^{\infty}$ (in three or more dimensions) and $\sigma >n/2:$
$$
\begin{aligned}
& \left\|\langle x\rangle^{-\sigma} e^{i \Delta(t-s)} f\right\|_{L^{2}} \leqslant\left\|\langle x\rangle^{-\sigma}\right\|_{L^{2}}\left\|e^{i \Delta(t-s)} f\right\|_{L^{\infty}} \\
& \leqslant\left\|\langle x\rangle^{-\sigma}\right\|_{L^{2}}|t-s|^{-n / 2}\|f\|_{L^{1}} .
\end{aligned}
$$

For $\sigma>\frac{n}{2}$ the first term is bounded.
For $n \geqslant 3$, the integral over $s$ is bounded.

In our case $f \equiv V(x, s) \psi(s) \in L^{1}$
for $V(x, s) \in L_{s}^{\infty} L_{x}^{2},$ which follows from our decay assumption on $V.$
\newline
\newline

{\bf The interval $[t-\varepsilon, t]$.}

We want to control the expression
$$
\langle x\rangle^{-\sigma} \int_{t-\varepsilon}^{t} e^{i \Delta(t-s)} \Delta V(x, s) \psi(s) d s .
$$
Integration by parts gives:
$$
\begin{gathered}
\int_{t-\varepsilon}^{t} \Delta e^{i \Delta(t-s)} V(s) \psi(s) d s=\int_{t-\varepsilon}^{t} i \partial_{s}\left(e^{i \Delta(t-s)} V \psi\right) d s \\
-i \int_{t-\varepsilon}^{t} e^{i \Delta(t-s)}\left(\frac{\partial V}{\partial s} \psi(s)+V(-i) i \frac{\partial \psi}{\partial s}\right) d s \\
=i V(t) \psi(t)-i e^{i \Delta \varepsilon} V(t-\varepsilon) \psi(t-\varepsilon) \\
-i \int_{t-\varepsilon}^{t} e^{i \Delta(t-s)} \frac{\partial V}{\partial s} \psi(s) d s \\
-i \int_{t-\varepsilon}^{t} e^{i \Delta(t-s)} V(-i \Delta-i V) \psi(s) d s .
\end{gathered}
$$

The first two terms are uniformly bounded in $L_{x}^{2}$. The third term times $\langle x\rangle^{-\sigma}$ gives
$$
\begin{aligned}
& \int_{t-\varepsilon}^{t}\left\|\langle x\rangle^{-\sigma} e^{i \Delta(t-s)} \frac{\partial V}{\partial s} \psi(s)\right\|_{L^{2}} d s \\
& \leqslant \int_{t-\varepsilon}^{t}\|\langle x\rangle^{-\sigma}\|_{L^{\infty}}
|\frac{\partial V}{\partial s} \psi(s)\|_{L_{s}^{\infty} L_{x}^{2}} d s, \\
& \leq \int_{t-\varepsilon}^{t}\left\| \frac{\partial V}{\partial s} \psi(s)\right\|_{L_{s}^{\infty} L_{x}^{2}} d s<c \epsilon \left\| \frac{\partial V}{\partial s}\right \|_{\infty} \|\psi\|_2 . \\
\end{aligned}
$$
 We now choose $\epsilon$ small enough so that $c \epsilon \left\| \frac{\partial V}{\partial s}\right \|_{\infty} <1/2 .$
The estimate of the last term is similar, where we have $V^{2} \psi(s)$ replacing $\frac{\partial V}{\partial s}$, except that the Laplacian term is bootstrapped:

$$
\begin{aligned}
& \left\|\langle x\rangle^{-\sigma} \int_{t-\varepsilon}^{t} e^{i H(t-s)} V \Delta \psi(s) d s\right\|_{L^{2}_x} \\
& \quad \leqslant c \varepsilon \|\langle x\rangle^{\sigma}V\|_{L^{\infty}}\left\|\langle x\rangle^{-\sigma} \Delta \psi(s)\right\|_{L_{s}^{\infty}(t-\varepsilon, t) L_{x}^{2}},
\end{aligned}
$$
 Hence we need $\left\|V\langle x\rangle^{\sigma}\right\|_{L^{\infty}_x}<c$.
\end{proof}

\section{Control for Large x}

In this section, we upgrade control of derivatives for large $x$.

\begin{proposition}

Under the previous assumptions and Local Smoothness Theorem, we have that
\begin{align}\label{p-bd}
&\|p \psi(t)\|_{L^{2}} \leqslant t^{a}, \quad t \gg 1\\
&1 / 5=a \quad \text{for} |V| \approx\langle x\rangle^{-3-\varepsilon} \text{and} |\nabla V|+\left|\partial_{t} V\right| \leqslant\langle x\rangle^{-6-\varepsilon}.
\end{align}
\end{proposition}
\begin{proof}
Let $F_2(\lambda \geq 1)$ be a smooth characteristic function of the interval $[1,\infty).$
We use the following PROB:
$$
\begin{gathered}
h(t)=F_{2}\left(\frac{|p|}{t^{\alpha}} \geqslant 1\right) H(t)+H(t) F_{2}\left(\frac{|p|}{t^{\alpha}} \geqslant 1\right) \\
H(t)=p^{2}+V(x, t) .
\end{gathered}
$$

Then,
$$
\begin{aligned}
\partial_{t} h(t)= & \left\langle i\left[V, F_{2}\right] H+H i\left[V, F_{2}\right]\right\rangle \\
& +\left\langle F_{2} \frac{\partial V}{\partial t}+\frac{\partial V}{\partial t} F_{2}\right\rangle \\
& +\left\langle-\frac{\alpha}{t} \tilde{F}_{2}^{\prime} H+H\left(-\frac{\alpha}{t}\right) \tilde{F}_{2}^{\prime}\right\rangle \equiv I_{1}+I_{2}+I_{3}
\end{aligned}
$$
$I_{2}$ has a localization as a result of the decay in $x$ of $\frac{\partial V}{\partial t}$.

Hence, we obtain:
$$
\begin{aligned}
I_{2}= & \left\langle F_{2}(\Delta+1)^{-1}(-\Delta+1){\frac{\partial V}{\partial t}}\right\rangle+c \cdot c \\
& t^{-2 \alpha}\left\langle\Delta \psi, \tilde{F}_{2} \frac{\partial V}{\partial t}\psi\right\rangle+c \cdot c \cdot= \\
= & t^{-2 \alpha}\left\langle\Delta \psi,\langle x\rangle^{-\sigma}\langle x\rangle^{\sigma} \tilde{F}_{2} \frac{\partial V}{\partial t} \psi\right\rangle+c \cdot c . \\
\leq & c t^{-2 \alpha}\left\|\langle x\rangle^{-\sigma} \Delta \psi\right\|_{2}\left\|\langle x\rangle^{\sigma} \frac{\partial V}{\partial t}\right\|_{L_{t, x}^{\infty}}\|\psi\|_{2}. \\
I_{3}= & -\frac{\alpha \cdot 2}{t}\left\langle\tilde{F}_{2}^{\prime} p^{2}\right\rangle+\left\langle-\frac{\alpha}{t}\left(\tilde{F}_{2}^{\prime} V+V \tilde{F}_{2}^{\prime}\right)\right\rangle \\
\leq & O\left(t^{-1+2 \alpha}\right)+O(t^{-1})
\end{aligned}
$$
For the estimates above, we used that $$ F_2(p^2+1)^{-1}=\mathcal{O}(t^{-2\alpha}) $$ and the previous local smoothness result.
We used that
$\langle x\rangle^{\sigma} \tilde{F}_{2} \frac{\partial V}{\partial t}$ is higher order in $x,t$, since by commuting powers of $x$ through $\tilde{F}_{2}^{\prime}$ we gain a power of $x$ and a power of $t^{-\alpha}$.

It remains to estimate $I_{1}$.
First, we note that the leading order of the commutator $i\left[V, F_{2}\right]$ is given by
$$
\frac{\partial V}{\partial x} \cdot \frac{\partial F_{2}}{\partial p}=\left(\frac{\partial V}{\partial x} \cdot  \tilde F_{2}\right) t^{-\alpha}; \quad \tilde F_2\equiv \tilde F_2(\frac{|p|}{t^{\alpha}}=1).
$$

Hence, as before, we can estimate
$$
\frac{\partial V}{\partial x} \cdot \frac{\partial F_{2}}{\partial p}  p^{2} \sim t^{\alpha}(-\Delta+i)(-\Delta+i)^{-1} \frac{\partial V}{\partial x} \cdot \tilde F_{2}(-\Delta+i)^{-1}(-\Delta+i)
$$

Assuming we have enough decay in $x$ for $\frac{\partial V}{\partial x}$, we gain from each side $t^{-2 \alpha}$.
So, all together, such terms decay like
$$
t^{-\alpha} \cdot t^{2 \alpha} t^{-4 \alpha}=t^{-3 \alpha}
$$
where $t^{-\alpha}$ comes from the commutator, $t^{2 \alpha}$ comes from $p^{2} \tilde{F}_{2}^{\prime}$,
$t^{-4 \alpha}$ comes from each factor of $(-\Delta+c)^{-1} \tilde{F}_{2}^{\prime}$.
For this we need that $\frac{\partial V}{\partial x}$ decays at least like $\langle x\rangle^{-4-\varepsilon}$.
We can use similar improvement of the estimates of $I_{2}$ and $I_{3}$. For example, to control $I_2,$ we multiply the $\psi$ on the left side of the scalar product by $<x>^{-\sigma}$ and move a $<x>^{\sigma}$, to the right up to $\frac{\partial V}{\partial t}.$
We then multiply and divide by the laplacian, on the left, and we gain another factor of $t^{-2\alpha}.$ 
This is where we need more decay of the potential and its time derivative.

we conclude that:

$\partial h(t) \sim t^{-3 \alpha} c\left(\left\|\psi_{0}\right\|_{H^1}\right)$.
Integrating both sides, and remembering that
$2\langle H(t)\rangle \sim \langle h(t)\rangle +O\left(t^{+2 \alpha}\right)$, we choose $\alpha$ such that
$$
t^{2 \alpha} \sim \int^{t} s^{-3 \alpha} d s \sim t^{-3 \alpha+1}
$$

This gives $2 \alpha=1-3 \alpha$ or $\alpha=1 / 5$. 

Consequently, we proved under the assumption of $\langle x\rangle^{-6-\epsilon}$ decay for the derivatives of $V$, that
$$
\left\langle \psi(t), p^{2}\psi(t)\right\rangle \leqslant t^{2 / 5}
$$
\end{proof}
\section { Incoming and outgoing Waves}
We begin with the proof, of general interest, that the incoming part of the solution is smooth (as $\mathrm{H}^{2}$ in our case).
Let as before, $A=\frac{1}{2}(x \cdot p+p \cdot x)$ and $M, R \gg 1$. 

Then we define by the spectral theorem the following operators \cite{soffer2011monotonic}:
$$
\begin{aligned}
& P^{+}(A)=P_{M, R}^{+}=\frac{1}{2}\left(I+\tanh \left(\frac{A-M}{R}\right)\right) \\
& P^{-}(A)=\frac{1}{2}\left(I-\tanh \left(\frac{A+M}{R}\right)\right) \\
& P^{0}(A)=I-P^{-}(A)-P^{+}(A)=\frac{1}{2}\left(\tanh \left(\frac{A+M}{R}\right)-\tanh \left(\frac{A-M}{R}\right)\right)
\end{aligned}
$$

\begin{proposition} [Local Decay-Incoming Waves]

\begin{equation}\label{LD2}
\left\|P^{-}(A) e^{i \Delta \tau} \chi(|p|=K) f(x)\right\|_{L^{2}} \leq\langle K \tau\rangle^{-2}\left\|\chi\langle x\rangle^{\sigma} f\right\|_{H^{2}}, \text{with} \,\, \tau>0,  K \tau \gg 1 .\quad  \sigma>2.
\end{equation}
\end{proposition}
\begin{proof}
$$
\begin{aligned}
& P^{-}(A) e^{i \Delta \tau} \chi f(x)=P^{-}(A) e^{i \Delta\tau}\left(A^{2}+1\right)^{-1}\left(A^{2}+1\right)\chi  f \\
& \left(A^{2}+1\right) \chi f=\tilde{\chi}\left(A^{2}+1\right)\chi f+\left[A^{2},\tilde \chi \right] f \\
& =\tilde \chi\left(\left(x \cdot\nabla_{x}\right)^{2}+\left(x \cdot \nabla_{x}+C\right)\right )\chi f+O(K^{-2})\tilde \chi^{\prime}(\Delta-1)\chi f.
\end{aligned}
$$

Using the assumption that $f \in H^{2},\langle x\rangle^{2} f \in H^{2}$ the RHS is bounded by
$$
c\| \chi f\|_{H^{2}}
$$
where $\chi=\chi(|p| \sim k), \quad \tilde\chi \chi =\chi$.
So, it remains to estimate
$$
P^{-}(A) e^{i \Delta \tau} \chi\left(A^{2}+1\right)^{-1}
$$

We now move $\left(A^{2}+1\right)^{-1}$ to the Left:
$$
\begin{aligned}
& {\left[\chi,\left(A^{2}+1\right)^{-1}\right]=-\left(A^{2}+1\right)^{-1}\left[\chi, A^{2}+1\right]\left(A^{2}+1\right)^{-1}} \\
& i\left[A^{2},\chi \right]=A i[A,\chi ]+i[A, \chi] A \\
& =A p \chi^{\prime}(p)+p \chi{\prime}(p) A=2 p \chi{\prime}(p) A \\
& +\left(\chi{\prime}(p)+p^{2} \chi^{\prime \prime}\right) A
\end{aligned}
$$

Therefore
$$
\begin{gathered}
\chi\left(A^{2}+1\right)^{-1}=\left(A^{2}+1\right)^{-1} \chi-\left(A^{2}+1\right)^{-1} G(p) A\left(A^{2}+1\right)^{-1} \\
G(p)=O(1) .
\end{gathered}
$$

So, it is sufficient to estimate
$$
\begin{aligned}
& P^{-}(A) e^{-i \Delta \tau}\left(A^{2}+1\right)^{-1} \chi(p)= \\
& =P^{-}(A)\left(\left(A-2 p^{2} \tau\right)^{2}+1\right)^{-1} e^{-i \Delta \tau} \chi(p)
\end{aligned}
$$

Formally, since $P^{-}(A)$ projects on $A \leq 0$ up to small corrections,then $A-2 p^{2} \tau<0$, so at the "symbol" level the result is clear.
However, $P^{-}(A)$ is not a $\psi$DO, and we cannot use symbol calculus.

First, we approximate $A-2 p^{2} \tau$ by using $P_c^-(A)=F(A/M \leq -1)$ 

to approximate $P^{-}(A)\left(\left(A-2 p^{2} \tau\right)^{2}+1\right)^{-1}$:
$$
\begin{aligned}
& {\left[\left(A-2 p^{2} \tau\right)^{2}+1\right]^{-1}-\left[\left(P_c^{-}(A) A-2 p^{2} \tau\right)^{2}+1\right]^{-1} } \\
= & {\left[\left(P_c^{-} A-2 p^{2} \tau\right)^{2}+1\right]^{-1}\left\{\left(P^{+} A+P^{0} A\right)\left(A-4 p^{2} \tau\right)+\left(P^{+} A+P^{0} A\right) A\right\} \times } \\
& {\left[\left(A-2 p^{2} \tau\right)^{2}+1\right]^{-1} }
\end{aligned}
$$

The term $P^{-}(A)\left[\left(P_c^{-} A-2 p^{2} \tau \right)^{2}+1\right]^{-1}$ satisfies the estimate we want, being bounded by
$$
\left[\left(1+2 p^{2} \tau\right)^{2}+1\right]^{-1} .
$$
\begin{proposition}[ Non-analytic]
The term 
\begin{equation}\label{N-An}
\left\langle\left[\left(P_c^{-} A-2 p^{2} \tau \right)^{2}+1\right]^{-1}\right\rangle 
\end{equation}
is bounded by 
\begin{equation}
c\left\langle \left [1+ (2p^2\tau)^2+ (P_c^-A)^2 \right ]^{-1}\right\rangle
\end{equation}

\end{proposition}

\begin{proof}
 We note that $\left(P_c^{-} A-2 p^{2} \tau \right)^{2}+1\geq 1+ (2p^2\tau)^2+ P_c^- A^2-4p\tau \tilde P_c^-Ap, $
where $\tilde P^- A$ stands for $pP^-Ap+[p,[p,P_c^-A]]$ but only if $P_c^-$ is an analytic function, that we can commute twice with $p.$

But it is not, and there is no simple formula for such commutators.

To proceed, we approximate the function defining $AP_c^-$ by an analytic function, and control the difference. The difference will still be of the form $p^2 \bar g(A)+ \bar g(A)p^2,$ but small and bounded.

To obtain an approximation which is analytic and can be commuted with $p$, we convolve the function in question with an analytic bump function, which is analytic in a strip wider than $\pi/2.$

For this we use the function given by $1/\cosh{\beta x},$ so we replace

$$
AP_c^-(A)\equiv G(A)\,\,  \text{by}\,\, g(A)=G(x)*[1/\cosh{\beta x}](y); \,\,  y=A.
$$
$g(A)$ is a smooth analytic function in a strip of width $(1/\beta) \pi/2$ around the real axis, not including the boundary.
 So we need to choose $\beta <\pi/2$ to be able to commute with $p.$
 We conclude that
 $$
 p^2G+Gp^2= 2p(g+g')p +p^2\bar g+ \bar g p^2.
 $$
 Now notice that the function $g$ has support also on $A>0.$ We need this part to be small. 
 
 Indeed, for $A$ large and positive, the convolution with $G$ vanishes exponentially for large positive $A$, at a rate $e^{-\beta A}.$
 Near $A=0$, we can make the size small, by making the function $G$ to vanish for $A> -L/\beta.$ Then the convolution will be bounded by $(L/\beta)e^{-L}$ for $A$ near zero. 
 Since $\beta \approx 1$ this factor is exponentially small in $L$.
 
 The factor $(L/\beta)$ comes from the multiplication by $A$ in the definition of $G.$
 
It remains to estimate the size of $\bar g.$
First, we need to normalize the convolution factor, so that its integral over all space is 1. This gives a factor $\beta/\pi.$

 Then, the difference $\bar g$ is given by
 $$
 \bar g(A)= \mathcal{F}^{-1}[(1-1/\cosh(\frac{\pi^2 k}{\beta}))\hat G(k).
 $$
 For $k$ near zero, the factor $(1-1/\cosh(\frac{\pi^2 k}{\beta}))\approx ck^2,$ exactly cancels the singularity at zero of the distribution $\hat G(k).$
 To make this quantity small in $L^{\infty}$ we need to make $L^1(dk)$ small. For this, we need to choose the support of $\hat G(k)$ close to zero. This is done by choosing $G(A)\sim G(\delta A)$ with small $\delta.$

We can now sum up the overall estimate
\begin{align}
&\left \langle \left(P_c^{-} A-2 p^{2} \tau \right)^{2}+1\right \rangle = \left\langle 1+ (2p^2\tau)^2+ (P_c^-A)^2 -2pg(A)p\tau-\tau^2 (p^2\bar g+\bar g p^2)\right\rangle\\
&=\left\langle 1+ (2p^2\tau)^2+ (P_c^-A)^2)\right\rangle -2\left\langle pg(A)F(A\leq 0)p\right\rangle-2\left\langle pg(A)F(A\geq 0)p\right\rangle-\left\langle p^2\bar g+ \bar g p^2\right\rangle \tau.\\
& \left\langle pg(A)F(A\geq 0)p\right\rangle \leq 2\epsilon \tau \|p\psi\|^2,\\
& \left\langle p^2\bar g+ \bar g p^2\right\rangle \tau \leq 2\epsilon \|\psi\| (\|p^2\psi\|^2\tau^2)^{1/2}.   
\end{align}
It then follows that
\begin{equation}
 \left\langle 1+ (2p^2\tau)^2+ (P_c^-A)^2\right\rangle\geq (1-0)\left [\left\langle 1+ (2p^2\tau)^2+ (P_c^-A)^2\right\rangle\right]
\end{equation}
\end{proof}

Control the difference between the resolvents
$$ 
A_{\tau}^2\equiv \left(A-2 p^{2} \tau\right)^{2} \text { and }\left(P_c^{-} A-2 p^{2} \tau\right)^{2} \text {, }
$$
we commute $P^{-}(A)$ through, and use that
$P^{+}(A) P^{-}(A)$ is exponentially small $P^{+}(A) P^{-}(A) \sim e^{-|A| / R}$ for $|A|$ large,
$P^{+}(A) P^{0}(A)$ is exponentially localized near $A=+M$.

So,

\begin{align}\label{Atau}
& P^{-}(A)\left\{\left[\langle A_{\tau}\rangle^{2}+1\right]^{-1}-\left[\left\langle P_c^{-}(A) A_{\tau}\right\rangle^{2}+1\right]^{-1}\right\} \\
& =\left(\langle A_{\tau}P_c^-\rangle^{2}+1\right)^{-1}\left[P^{-}(A)\left(P^{+}+P^{0}\right) 4Ap^{2} \tau+\text { h.o.t. }\right]\left(\left\langle A_{\tau} \right\rangle^{2}+1\right)^{-1} \\
& +\left[P^{-}(A),\left[\langle A_{\tau}P_c^-\rangle^{2}+1\right]^{-1}\right]\left[\left\langle P^{+}+P^{0}\right) 4Ap^{2} \tau+\text { h.o.t. }\right]\left(\left\langle A_{\tau} \right\rangle^{2}+1\right)^{-1} \\
& =\left(\langle A_{\tau}P_c^-\rangle^{2}+1\right)^{-1} O\left(e^{-|A-M|/R} M\right) 4 p^{2} \tau\left(\left\langle A_{\tau} \right\rangle^{2}+1\right)^{-1} \\
& \left.+\left[P^{-}(A),\left[\langle A_{\tau}P_c^-\rangle^{2}+1\right]^{-1}\right]\left[\left(P^{+}+P^{0}\right) 4Ap^{2} \tau+\text { h.o.t. }\right]\left(\langle A_{\tau}\right\rangle^{2}+1\right)^{-1} .
\end{align}

We note that $\left(P_c^{-} A-2 p^{2} \tau\right)^{2}+1\geq (2p^2\tau)^2+1.$
 Here, $h.o.t$ stands for higher order terms, which come from terms containing the product $P^+(A)P^-(A),$
 for example $P^-(A)A^2(1-(P_c^-)^2). $

The first term on the RHS is bounded by
$$
O\left(e^{-|A -M| / R} M\right) \min\left\{\frac{1}{1+p^{2} \tau}+\frac{1}{1+\langle P^{-} A\rangle}\right\}
$$

The second term on the RHS is more difficult to control, since commuting a projection is creating an error term which is not a projection on the other side.
This is the reason why one has to use a special form of the projection; see \cite{soffer2011monotonic}. To this end, we use the following identity:
Let $A$ be the dilation generator, and
$$
U_{\theta}=e^{-i \theta A}
$$
the group of dilations on $L^{2}\left(\mathbb{R}^{n}\right),\quad \theta$-real.
Then for a general  operator $B$ on $L^{2}\left(\mathbb{R}^{n}\right),$  we define
$$
U_{\theta} B U_{-\theta} \equiv B_{\theta}.
$$

We will now extend these identities to the complex $\theta=\theta_{0}+i \theta_{1} \quad\left|\theta_{1}\right|<\delta$, provided that $B$ is a "dilation analytic", in the sense that $B_{\theta}$ makes sense on some common domain $D_{\theta} \subset L^{2}$, for all $\operatorname{such} \theta$.

We then have
\begin{proposition}[\cite{soffer2011monotonic}]
Let $B$ be an operator such that $B_{\theta}$ is dilation-analytic in $|\Im \theta|<\delta.$
Then,
\begin{equation}\label{P-}
    i[P^-(A),B]=\cosh^{-1}\frac{(A+M)}{R}[B_{-i/R}-B_{i/R}]\cosh^{-1}\frac{(A+M)}{R}
\end{equation}
for all $R>c=\mathcal{O}(1).$
\end{proposition}

Now we can control the commutator term with $P^-(A).$
To use the above proposition, we will need to know that  $\langle AP_c^- -2p^2\tau\rangle^{-2}$ is dilation analytic for small $\beta.$
$$
e^{\beta A}(2p^2\tau +1)^{-2}e^{-\beta A}=(2e^{2\beta i}p^2\tau+1)^{-2}.
$$
$$
\tau e^{2\beta i}p^2+1=\tau \cos{2\beta}p^2+1+i\sin{2\beta}p^2\tau,
$$
and for $\beta$  sufficiently small, the real part is larger than 1, for $\tau\geq 0.$

Then,
\begin{align*}
& 1+(AP_c^- - \tau e^{2\beta i}p^2)^2=1+(AP_c^- -2p^2\tau)^2\\
&-\epsilon(\beta)[AP_c^- p^2\tau+p^2\tau AP_c^-]+4\cos{4\beta}p^4\tau^2\\
&+i[4AP_c^-\sin{2\beta}p^2\tau]+4\sin{4\beta}p^4\tau^2]\neq 0.
\end{align*}
Here $\epsilon(\beta)=2-2\cos{2\beta}>0$ for all $\beta \neq 0$ small.

The above expression is a mixture of real and complex unbounded operators. In the Appendix we give a detailed proof of the dilation analyticity of such objects.

Consequently, we can use the above proposition: the commutator term with $P^-(A)$ in \eqref{Atau} using equation \eqref{P-},
is controlled by the norm of 
\begin{align*}
& \cosh^{-1}{(\beta A)}(P^-+P^0)Ap^2\tau\left(\langle A_{\tau}P_c^-\rangle^2+1\right)^{-1}\\ 
&\leq \frac{M}{\cosh{(\beta A)}}\left(\frac{1}{p^2\tau+1}+\frac{1}{\langle AP^-\rangle +1}\right ).
\end{align*}
Finally, we note that the resulting estimate of
$$
P^-(A)[\langle A_{\tau}\rangle^2+1]^{-1}\lesssim   \left(\frac{1}{p^2\tau+1}+\frac{1}{\langle AP^-\rangle +1}\right )
$$
can be used to control terms linear in $A$, since they are multiplied from the left by $P^-(A)[\langle A_{\tau}\rangle^2+1]^{-1}.$
This gives the final statement of the proposition. \end{proof}
We can now use the proposition on Local Decay of Incoming Waves \eqref{LD2} to control the Duhamel term in $H^{2}$, on the range of $P^{-}(A)$.
Since the initial data are in $H^{2}$, we only need to control the Duhamel term.

Integrating by parts, we get
$$
\begin{aligned}
-i \Delta & \int_{0}^{t-\varepsilon} e^{i \Delta(t-s)} V(s) \psi(s) d s \\
= & e^{i \Delta \varepsilon} V(t-\varepsilon) \psi(t-\varepsilon)-V(t=0) \psi(t=0) \\
& -\int_{0}^{t-\varepsilon} e^{i \Delta(t-s)}\left(\frac{\partial V}{\partial s} \psi(s)+i V \Delta \psi(s)-i V^{2} \psi(s)\right) d s
\end{aligned}
$$

So, we only need to show that the RHS is uniformly bounded in $L^{2}$.
The first and second terms are bounded in $L^{2}$ if $\langle x)^{m} V(x, s)$ is bounded from $H^{2} \rightarrow L^{2}$, since we proved that $\langle x\rangle^{-\sigma} \psi(s)$ is in $H^{2}$. Similarly $\frac{\partial V}{\partial s} \psi-i V \Delta \psi-i V^{2} \psi(s)$ are in $L^{2}$. So, we need to prove the integrability in $s_{\text {. }}$
For this, we use
$$
\left\|P^{-}(A) e^{i \Delta(t-s)} g(|p|\geq \delta)\langle x\rangle^{-\sigma}\right\| \leqslant c\langle t-s\rangle^{-2},
$$
for $\delta>0$.

We do not have the case $|p|<\delta$, since in this region the solution is in $\mathrm{H}^{s}$ for all $s>0$.
The integral in the interval $s \in(t-\varepsilon, t)$ is trivially bounded in $L^{2}$. We conclude:

\begin{theorem} {Incoming Waves}

If the initial data is in $\mathrm{H}^{2}$, then
\begin{equation}\label{In-Reg}
P^{-}(A) \psi(t) \in H^{2} \text {, uniformly in } t \text {. }
\end{equation}
\end{theorem}

\section{More On Incoming Waves}

As we saw in the previous section, most of the correction to the free flow is outgoing. We will further control the Incoming Waves by PROBs.
To control the incoming waves, we need a PRES that registers the fact that solutions become outgoing for large times.
The first PROB we use is $F\left(\frac{A+M}{R} \leq-1\right)$.

\begin{proposition}

The $H^{1}$ norm is bounded for all times, on average, on the negative spectral part of $A$ (incoming).
\end{proposition}
\begin{proof}

The proof uses the basic calculus of using analytic functions of self-adjoint operators, in particular, functions of $A,$ the dilation operators.
below, we will use an analytic function $F(z)$ in a strip of width (say) 1 around the $x$ axis. For example, $\tanh z.$
Using such functions we construct approximate projections.
The key computational tool we need is the commutator of such functions with other operators.
In particular, we have
$$
i[F(A/R),p]=i(F(\frac{A-i}{R})-F(A/R))p=-ip(F(\frac{A+i}{R})-F(A/R)).
$$
$$
e^{-itA}pe^{itA}=e^t p; \quad e^{-itA}xe^{itA}=e^{-t} x.
$$

\begin{equation}
\partial_{t}\left\langle F\left(\frac{A+M}{R} \leqslant-1\right)\right\rangle=\frac{1}{R}\left\langle p \psi, \tilde{F}^{\prime} p \psi\right\rangle+\langle i[V, F]\rangle 
\end{equation}

The first term is negative, so we get from above equation,

\begin{align}
\langle F\rangle_{T} & -\langle F\rangle_{0}+\frac{1}{R} \int_{0}^{T}\langle p \psi,| \tilde{F}\left(\frac{A+M}{R} \sim-1\right)|\, p \psi\rangle d s \\
& \lesssim\left|\int_{0}^{T}[V, F] d s\right| 
\end{align}

We get an estimate of the first derivative in the region $A \sim-M$ (in an interval of size $R$ ), in terms of the potential.
Next, we show that the potential term is of higher order in $\left(\frac{1}{R}\right)$ :
Write $V$ as $\langle x)^{-\sigma}\langle x\rangle^{\sigma} \vec{V}\langle x\rangle^{\sigma}\langle x\rangle^{-\sigma} \equiv\langle x\rangle^{-\sigma} \tilde{V}\langle x\rangle^{-\sigma}$.
Then

\begin{align*}
& {[V, F]=\left[\langle x\rangle^{-\sigma}, F\right] \tilde{V}\langle x\rangle^{-\sigma}+\langle x\rangle^{-\sigma} \tilde{V}\left[\langle x\rangle^{-\sigma}, F\right]} \\
& +\langle x\rangle^{-\sigma}[\tilde{V}, F]\langle x\rangle^{-\sigma} \\
& =\frac{1}{R}\langle x\rangle^{-\sigma} \tilde{F}^{\prime} \tilde{V}\langle x\rangle^{-\sigma}+\langle x\rangle^{-\sigma} \tilde{V} \tilde{F}^{\prime}\langle x\rangle^{-\sigma} \frac{1}{R} \\
& +\langle x\rangle^{-\sigma}[\tilde{V}, F]\langle x\rangle^{-\sigma} \tag{5,3}
\end{align*}

Here, we used $[x, F]=\tilde{F}^{\prime} x=x \tilde{F}^{\prime}$ where $\tilde{F}^{\prime}$ stands for a discrete derivative of $F$ in the $\pm i$ direction.

We get another power of $\left(\frac{1}{M}\right)\left(<\frac{1}{R}\right)$ using
$$
F(A+i)^{-1} \sim \frac{1}{M} F
$$

Then, write $F=F(A+i)^{-1}(A+i)$
$$
=\tilde F \frac{1}{M}(x \cdot p+c).
$$
We can write all terms as
$$
\frac{1}{M R}\left\langle p\langle x\rangle^{-\sigma} \psi, O(1) p\langle x\rangle^{-\sigma} \psi\right\rangle .
$$

In particular, the term, by the Commutator Expansion
Lemma,
$$
\begin{aligned}
& \langle x\rangle^{-\sigma}[\tilde{V}, F]\langle x\rangle^{-\sigma}= \\
& =\langle x\rangle^{-\sigma}(x \cdot \nabla \tilde{V}) \cdot \tilde{F}^{\prime} \frac{1}{R}\langle x\rangle^{-\sigma} \\
& +\langle x\rangle^{-\sigma} \frac{1}{R^{2}} O_{2}(F)\langle x\rangle^{-\sigma} \\
& =\frac{1}{R}\langle x\rangle^{-\sigma}(x \cdot \nabla \tilde{V})(A+i)(A+i)^{-1} \tilde{F}^{\prime}\langle x\rangle^{-\sigma} \\
& +\langle x\rangle^{-\sigma} \frac{1}{R^{2}} O_{2}(F)\langle x\rangle^{-\sigma} \\
& =\frac{1}{M R} P\langle x\rangle^{-\sigma}(x \cdot \tilde{V})\langle x\rangle^{-\prime}\langle x\rangle^{-\sigma} \\
& +\langle x\rangle^{-\sigma} \frac{1}{R^{2}} \mathcal{O}_{2}(F)\langle x\rangle^{-\sigma} .
\end{aligned}
$$
We need to know that $\mathcal{O}_2(F)$ is bounded.
\begin{equation}
\mathcal{O}_2(F)=\int{\hat F(\lambda)}e^{-i\lambda A}d \lambda \int_0^\lambda ds e^{isA} \int_0^sdu e^{-iuA}(x\cdot \nabla)^2\tilde Ve^{iuA}. 
\end{equation}
So, an immediate bound in $L^2$ is obtained if $\|x^2\Delta \tilde V \|_{L^{\infty}}<c<\infty.$

But in fact, a weaker condition is sufficient:
\begin{equation}
  \int_0^sdu e^{-iuA}(x\cdot \nabla)^2\tilde Ve^{iuA}\langle p\rangle^{-1}    
\end{equation}
be bounded.
Using the scaling $e^{-iuA}re^{iuA}=e^{-u}r,$
The above integral reduces to 
\begin{equation}
\int_r^{e^{-s}r} dy \frac{f(y)}{y}\langle p\rangle^{-1}; \quad f(z)\equiv (z\cdot \nabla_z)^2\tilde V(z,t).    
\end{equation}
The singularity in the integral near $y=0$ is $r^{-m-1}.$
 By Sobolev type estimate, we know that $\langle p\rangle^{-1}$ bounds $\frac{1}{|x|}.$
 A factor of $1/r$ is then bounded by $\langle p\rangle^{-1}. $ 

The singularity near zero is reduced to that of  $|z|^2(\Delta \tilde V(z)).$
 Up to this point we have an estimate for a fixed $A\sim M.$


{\bf We now iterate this estimate to higher Sobolev norm.}

We use as PROB
\begin{equation}
\left\langle p \psi(t), A F\left(\frac{A+M}{R} \leq-1\right) p \psi\right\rangle \leq 0 .
\end{equation}

In this case, the RHS, similar to previous estimate, is given by
$$
\begin{aligned}
 \quad & \int_{0}^{T}\left\langle p^{2} \psi(t), \tilde{F'}\left(\frac{A+M}{R} \leqslant-1\right) p^{2} \psi(t)\right\rangle d t \\
& +\int_{0}^{T}\langle i[V, p A F p]\rangle d t
\end{aligned}
$$

The estimate of the potential is as before. The commutator has two derivatives w.r.t. $x$. Hence, we can pull $(x)^{-\sigma}$ on both sides and one more derivative on each side; then we can use that locally in $x$, the $H^{2}$ norm is uniformly bounded.
Therefore, again the potential term is bounded by $T.$ Notice that the LHS is integrated on $\left[0, T^{\prime}\right]$ s.t. the LHS is bounded for such $T^{\prime}$, using the previous estimate.

{\bf Upgrading the Estimates}
  We will sum over dyadic intervals covering the negative real line, starting from $-M_0.$
 Multiplying by $\left(R_{0} 2^{n}\right)^{2}( n)^{-1}(n \geqslant 1)$,
using for each $n=0,1,2,3...,$  \quad $F_n(\frac{A+M_0+2^nR_0}{2^nR_0}\leq -1).$
After multiplying by $R_0^2 2^{2n}/n$ we take the sum over $n.$

The resulting inequality becomes:
\begin{align}
&\langle\psi(T),A^2F(\frac{A+M_0}{R_0}\leq -1)\psi(T)\rangle\\
&-\langle\psi(0),A^2F(\frac{A+M_0}{R_0}\leq -1)\psi(0)\rangle\\
&+\int_0^T\langle p\psi(s),|F(\frac{A+M_0}{R_0}\leq -1)A\ln^{-1}{\langle A\rangle}|p\psi(s)\rangle ds\\
&\approx\int_0^T\langle p\psi(s)\langle x \rangle^{-\sigma} F(\frac{A+M_0}{R_0}\leq -1)\ln^{-1}{\langle A\rangle}\langle x \rangle^{-\sigma}p \psi(s)\rangle ds\\
&\leq c\|\psi(0)\|^2_{H^2} T.
\end{align}

Here, we use that 
\begin{align}
&\sum_n F_n\geq   F(\frac{A+M_0}{R_0}\leq -1)\\
& \sum_n 2^{kn}R_0^kF_n \lessgtr A^k  F(\frac{A+M_0}{R_0}\leq -1) \ln{|A|}.
\end{align}
 See \cite{liu2025large}.
 We therefore conclude that the {\bf average over time} of 
 \begin{align}
&\langle p\psi(s),|A|\ln^{-1}{\langle A\rangle}F(\frac{A+M_0}{R_0}\leq -1)p\psi(s)\rangle\\
\end{align}
is bounded.

{\bf Upgrading to pointwise estimate}

To obtain estimates that are better than those bound by $T$, we need to estimate the potential term in terms of the leading term on the RHS of the PRES.

The leading order term is of the form:
$$
\int_{0}^{T}\left\langle p \psi, \frac{A}{\ln \langle A \rangle} F\left(\frac{A+M_{0}}{R} \leq-c\right) p \psi\right\rangle_{s} d s
$$
(This is derived by multiplying by $R 2^{n} / n$ and summing).

The potential term (Assuming first that $V$ is dilation analytic) is:
 $$\left.\int_{0}\langle p \psi, x \cdot \tilde{V}|x|^{2} F_{A}\left\langle\frac{A+M_{0}}{R} \leqslant-c\right) p \psi\right\rangle d s$$
Hence $i\left[V, F_{A, M}\right]=c \frac{1 / R}{\cosh \left(\frac{A+M}{R}\right)} x \cdot \tilde{V} \frac{1}{\cosh \left(\frac{A-M}{R}\right)}$
with
$$
x\cdot \tilde{V} \equiv V_{i R^{-1}}-V_{-i R^{-1}}=V\left(\frac{i|x|}{R}\right)-V\left(\frac{-i|x|}{R}\right)
$$

Then, we calculate (using the notation $\operatorname{ch} \equiv \cosh \left(\frac{A+M}{R}\right)$)
$$
\begin{aligned}
& \left(\frac{1}{\cosh \left(\frac{A+M}{R}\right)} x \cdot \tilde{V} \frac{1}{\cosh \left(\frac{A+M}{R}\right)}\right)= \\
= & \frac{1}{\operatorname{ch}} \frac{1}{A+i}(A+i) x \cdot \tilde{V}(A+i) \frac{1}{A+i} \frac{1}{ch} \\
= &
\end{aligned}
$$
$$
\begin{aligned}
= & c \frac{1}{(i+A) ch} x\cdot \tilde{V}|x|^2\frac{1}{(i+A) ch}+\frac{p}{(i+A) ch} x\cdot \tilde{V}|x|^2 \frac{1}{(i+A) ch} \\
& +\left\{\frac{1}{\{(i+A) c h}\left[p, A\, ch\right] \frac{1}{(i+A) c h} x\cdot \tilde{v}|x|^{2} \frac{1}{(i+A) ch}+\text { similar }\right\}
\end{aligned}
$$

Using $\frac{1}{\operatorname{ch}^{2}}=\frac{1}{\cosh ^{2}\left(\frac{A+M}{R}\right)}$ localizes 
$A$ around $-M$ with width $R$, multiplication by $R^{2} 2^{n+2} / n$ of the quantity
$$
\frac{1}{\cosh ^{2}\left(\frac{A+M+2^{n} R}{R}\right)},
$$
 and summing over $n=0,1,2 \ldots$,
this sum is essentially
$$
(\ln \langle A\rangle)^{-1}(A+i)^{-2} A^{2} F_{A}\left(\frac{A+M}{R} \leq-C\right) .
$$

So, we project on $\left|\frac{A}{\ln \langle A\rangle}\right|>K \gg\left\||x|^{2} x \cdot \tilde{V}\right\|_{L^{\infty}}$ 
and we find that the leftover term is
$$
\int_{0}^{T}\left\langle p \psi, F_{A} F_{1} E(|A| \leqslant K)\right||x|^{2} x \cdot \tilde{V}|p \psi\rangle_{s} d s
$$

But for $K \ll M_{0}$, we have that
$$
F_{A}\left(\frac{A+M_{0}}{R} \leq-c\right) F_{A}(|A| \leq K) \approx e^{-M_{0}} .
$$

So, the leftover term is bounded by
$$
c e^{-M_{0}} T.
$$

We therefore get the following estimate

\begin{align}\label{income}
& \left\langle\psi(t), {A}^{2} F\left(\frac{A+M_{0}}{R} \leq-C\right) \psi(t)\right\rangle_{T}-\left\langle\psi(0), A^{2} F_{A} \psi(0)\right\rangle_{0} \\
& =\int_{0}^{T}\left\langle p \psi(s),(\ln \langle A\rangle)^{-1}(A-\tilde{W}(x)) F(|A|>K) F_{A} p \psi(s)\right\rangle d s \\
& +\int_{0}^{T}\left\langle p \psi(s), F_{A} \tilde{W} F(|A| \leq K) p \psi(s)\right\rangle d s+h .o. t
\end{align}

The last term on the RHS is bounded by
$$
c e^{-M_{0}} T,
$$
since
$$
F_A \tilde W(x) E(|A|\leq K)= \tilde W(x) E(|A|\leq K) F_A+(1/ch)(1/R)(\tilde W)_{\beta} (1/ch) E(|A|\leq K),
$$
and both terms on the RHS are bounded by $e^{K-M_0-CR}.$
Here $(\tilde W)_{\beta}$ is the dilation transformation by the parameter $\beta=1/R.$
The first term on the RHS is negative since
$$
(|A|-\tilde{W}) F_{K}(A>K) \sim \tilde{F}_{K}(|A|-\tilde{W}) \tilde{F}_{K} \geqslant \theta \tilde{F}_{K}^{2}
$$
where $\tilde{F}_{K}^{2}=F_{K}$ and we used that
$$
K \gg \||x|^{2} x \cdot \tilde{V} \|_{L^{\infty}}.
$$

The h.o.t. come from commutators. Such commutators are of higher order in $(1 / R)$ and/or $A,$ as
$[p, A] \sim p ;[x, F(A / R)] \sim x \tilde{F}^{\prime} / R$.

Up to this point, we used dilation analyticity of the potential to write 
$$
i[V,F]\sim \frac{1}{ch} x\cdot\tilde {V}\frac{1}{ch}.
$$
However, it is not necessary:
\begin{align}
&\langle \psi, i[V,F_A]\psi\rangle=\langle \psi, iVF_A\psi\rangle -c.c.\\
&=\langle \psi, iVF_AG_A(|A|\leq M_0/2)\psi\rangle+\langle \psi, iVF_AG_A(A\leq -M_0/2)\psi\rangle\\
&=O(e^{-M_0/2})+c\langle p^2\psi, V|x|^4 F_A(A+i)^{-4}p^2\psi\rangle\\
&+c\langle \psi, \tilde{V} F_A(A+i)^{-4}\psi\rangle\\
&+O(e^{-M_0/2})\leq c\|\langle x\rangle^{-\sigma}p^2\psi\| \|\langle x\rangle^{-\sigma} F_A p^2\psi\|M_0^{-4} +O(e^{-M_0/2})
\end{align}

The integral over time is then controlled by
\begin{align}
&cTe^{-M_0/2}\\
&+cM_0^{-4}\int_1^T\| \langle x\rangle^{-\sigma}p^2\psi\|_{L_t^{\infty}L^2_x} \| \langle x\rangle^{-\sigma}F_Ap^2\psi\|_{L_t^{\infty}L^2_x} ds.
\end{align}

 If we now choose $M_{0} \sim t^{\alpha}$, we get the bound
$$
\begin{aligned}
& \sim\left(\int_{1}^{T} s^{-4 \alpha} d s\right)\left\|\langle x\rangle^{-\sigma} F_{A} p^{2} \psi\right\|_{L_{s}^{\infty} L_{x}^{2}}^2 \\
& \leq C
\end{aligned}
$$
where $C<\infty$, for
$$
\int_{1}^{T} s^{-4 \alpha} d s<\infty \quad(\alpha>1 / 4) .
$$

We conclude that the LHS of (\eqref{income}) is also bounded:
$$
\left\langle F_{A}\left(\frac{A+t^{\alpha}}{R} \leq-1\right) A^{2}\right\rangle \approx 1 .
$$

We conclude that if $|x| \geq 1$, then $\left.\left.\left\langle F_{A}\right| p\right|^{2}\right\rangle \leq 1$.
It remains to show that
$$
\left\|F_{A}\left(\frac{A+t^{\alpha}}{R} \geqslant-1\right) \Delta \psi(t)\right\|_{2} \sim 0(1)
$$
\end{proof}
\begin{remark}
Since we use an analytic PROB, we do not have symmetrization error terms. Hence, all error is of order $M_0^{-4}.$
Therefore, multiplying the entire estimate by $M_n^4$ and summing over $n$, we get a bound on $A^{6-0},$ which for $|x|>1$ also is a bound on the norm $H^{5/2-0}$. We see the expected extra smoothing in the region away from the Propagation Set.
\end{remark}
\section{Control of the solution on the Propagation Set}

High frequency and long-time parts of the solution of the free flow concentrate on the Propagation Set:
$$
P S_{E}=\left\{\left.\frac{x}{t}=2 p \right\rvert\, \quad p^{2}=E\right\}
$$
for the part of the solution localized at energy $E.$
Sharper localization of the solution is constructed by projection on the region $|x-2 p t| \leq t^{\alpha}$, $0 \leqslant \alpha<1$.

Consider the following $P R O B$ :

\begin{align}\label{HF}
& \partial_{t}\left\langle\psi(t),\left(H F_{c}+F_{c} H\right) \psi(t)\right\rangle=\left\langle\psi(t),\left(\frac{\partial V}{\partial t} F_c+F_{c} \frac{\partial V}{\partial t}\right) \psi(t)\right\rangle \\
& +\left\langle\psi(t),\left(H \frac{\partial F_{c}}{\partial t}+\frac{\partial F_{c}}{\partial t} H\right) \psi(t)\right\rangle \\
& +\left\langle\psi(t)\left[H i\left[V, F_{c}\right]+i\left[V, F_{c}\right] H\right] \psi(t)\right\rangle \\
& \text { Here } H=H_{0}+V(t)=p^{2}+V(t) \\
& F_{c}=F_{c}\left(\frac{|x-2 p t|}{t^{\alpha}} \leq 1\right)
\end{align}

\begin{proposition}
 The PROB mentioned above (\eqref{HF}) is uniformly bounded in time.
\end{proposition}

\begin{proof}
First, we use $F_{c}=F_{c}\left(\frac{|x-2 p t|}{s} \leq 1\right.)$

Then, the first term on the right of \eqref{HF} is of the form
$$
\sim\left\langle\langle x\rangle^{-\sigma} F_{c}\right\rangle+c \cdot c .
$$

The second term is zero:
$$
D_t F_{c}=i\left[-\Delta, F_{c}\right]+\frac{\partial F_{c}}{\partial t}=0 \quad \textbf{since} \, D_t|x-2pt|=0.
$$

The third term is of the form
$$
\left\langle p^{2} \tilde{V} \tilde{F}_{c}\right\rangle+c.c.
$$

So, we need to prove the integrability in time of the above term (depending on $s$ ).

For $\sigma=2$, we get

\begin{lemma}
    In the sense of quadratic forms, we have

$$
 \Re \langle x\rangle^{-2}\langle 2 p t\rangle^{2} F_{c} \lesssim c F_{c}+O\left(\frac{s}{\langle x\rangle} F_{c}\right)+O\left(\frac{s^{2}}{\langle x\rangle^{2}} F_{c}\right),
$$
where $2\Re B\equiv B+B^*.$
\end{lemma}
\begin{proof}
   
$$
\begin{aligned}
&\langle x\rangle^{-2}\left(4 p^{2} t^{2}\right) F_{c} \leq\langle x\rangle^{-2}\left(4 p^{2} t^{2}-x^{2}+x^{2}\right) F_{c} \\
& \leq F_{c}+\langle x\rangle^{-2}\left((2 p t-x) \cdot(2 p t+x)+x^{2}-2 p t \cdot x\right. \\
&-x \cdot 2 p t) F_{c} \\
& \leq F_{c}\langle x\rangle^{-2}\left(s(2 p t-x)+2|x| s+x^{2} c+|x| s\right) F_{c} \\
& \leq F_{c}\left[\left(s^{2} /\langle x)^{2}\right)+s /\langle x\rangle+1\right]
\end{aligned}
$$
\end{proof}
So,
$$
\begin{aligned}
& \langle x\rangle^{-2} F_{c}=\langle x\rangle^{-2}\left(1+4 p^{2} t^{2}\right)^{-1}\left(1+4 p^{2} t^{2}\right) F_{c} \\
& =\left(1+4 p^{2} t^{2}\right)^{-1}\langle x\rangle^{-2}\left(1+4 p^{2} t^{2}\right) F_{c} \\
& \quad-\langle x\rangle^{-2}\left[\langle x\rangle^{2},\left(1+4 p^{2} t^{2}\right)^{-1}\right]\langle x\rangle^{-2}\left(1+4 p^{2} t^{2}\right) F_{c}
\end{aligned}
$$

The first term on the RHS is bounded by using $A=x\cdot (p-x/(2s))+x^2/(2s)$ and using 
$$
\begin{aligned}
& \left( 1+4 p^{2} t^{2}=1+|x-2 p t|^{2}-2 x^{2}+8 A\right) \\
& \left(1+4 p^{2} t^{2}\right)^{-1}\left[2+\frac{s}{\langle x\rangle}+\frac{s^{2}}{\langle x\rangle^{2}}\right] F_{c}
\end{aligned}
$$

The second term is also bounded by
$$
\left(1+4 p^{2} t^{2}\right)^{-1}\left[1+\frac{s}{\langle x\rangle}+\frac{s^{2}}{\langle x\rangle^{2}}\right] F_{c}
$$
since
$$
\begin{gathered}
\langle x\rangle^{-2}\left[x^{2},\left(1+4 p^{2} t^{2}\right)^{-1}\right]\langle x\rangle^{-2}= \\
\langle x\rangle^{-2}\left(1+4 p^{2} t^{2}\right)^{-1}\left[x^{2}, 4 p^{2} t^{2}\right]\left(1+4 p^{2} t^{2}\right)^{-1}\langle x\rangle^{-2} \\
-i\left[x, 4 p^{2} t^{2}\right]=8 p t^{2}-\left[x\left[x, 4 p^{2} t^{2}\right]\right]=8 t^{2}
\end{gathered}
$$

Then
$$
\langle x\rangle^{-2}\left(1+4 p^{2} t^{2}\right)^{-1}\left(x \cdot p t^{2}+t^{2}\right)\left(1+4 p^{2} t^{2}\right)^{-1}\langle x\rangle^{-2}
$$

We need to control $p t^{2}$ and $t^{2}$.
$$
\begin{aligned}
& p t^{2}\left(1+4 p^{2} t^{2}\right)^{-1}\langle x\rangle^{-2}= \\
& \quad p t^{2}\left(1+4 p^{2} t^{2}\right)^{-1}|p| \frac{1}{|p|^{2}}\langle x\rangle^{2}=O(1) \\
& \langle x\rangle^{-2}\left(1+4 p^{2} t^{2}\right)^{-1}t^{2}\left(1+4 p^{2} t^{2}\right)^{-1}\langle x\rangle^{-2} \\
& =\langle x\rangle^{-2}|p|^{-1}|p|\left(1+4 p^{2} t^{2}\right)^{-1} t^{2}\left(1+4 p^{2} t^{2}\right)^{-1}|p| \frac{1}{|p|}\langle x)^{-2} \\
& =\langle x\rangle^{-1}O(1)\left(1+4 p^{2} t^{2}\right)^{-1}\langle x\rangle^{-1} .
\end{aligned}
$$

Next, we observe that for $|p| \geqslant 1$
$$
\begin{gathered}
|x-2 p t| \leq s \\
\text { with } s \equiv t^{\alpha}, \alpha \leq 1
\end{gathered}
$$
imply, at the symbol level, that for $t \gg t_{0}$,
$\langle x\rangle \gtrsim t$.
Hence, we expect, as we have used before, that $\quad\langle x\rangle^{-1} s F_{c}$ and $\langle x\rangle^{-2} s^{2} F_{c}$
with $t$ large, and $s \sim t^{\alpha}$, that these expressions are bounded.
To conclude
$$
\begin{aligned}
& \left\langle H(t) F_{c}+F_{c} H(t)\right\rangle_{T} \lesssim\left\langle H(0) F_{c}+F_{c} H(0)\right\rangle \\
& \quad+\int_{0}^{T}\left(1+t^{-2} O(1)\left\{\left\langle \psi, p^{2}\langle x)^{-\sigma+2}\langle x\rangle^{-2} \tilde{F}_{c} \psi\right\rangle+c . c\right\} d t\right.
\end{aligned}
$$

Finally, we write the $p^{2}$ factor as
$$
\begin{aligned}
& \left(4 p^{2}-\frac{x^{2}}{t^{2}}\right)+\frac{x^{2}}{t^{2}} \\
& 4 p^{2}-\frac{x^{2}}{t^{2}}=(2 p-x / t) \cdot(2 p-x / t+2 x / t)= \\
& =[2 p t-x) \cdot(2 p t-x)+(2 p t-x) \cdot 2 x] \frac{1}{t^{2}}
\end{aligned}
$$
we then bound the $(2 p t-x)$ factors by $s$, so the overall bound is (on support $\langle x\rangle^{-\sigma+}F_{c}$ )
$$
s^{2} / t^{2}+\frac{s\langle x\rangle}{t^{2}} \leqslant \frac{s^{2}}{t^{2}}+\frac{s}{t^{2}}
$$

So, by our choice of $s=T^{\alpha}$, we conclude that
$$
\left\langle H(T) F_{c}(T)+F_{c}(T) H(T)\right\rangle \stackrel{\varepsilon}{\sim}\left\langle H(0) F_{c}+F_{c} H(0)\right\rangle .
$$

\end{proof}

\section{Away From the Propagation Set}
\begin{theorem}[Away from PS]
We have the following bound away from the PS defined as $|x-2pt|\leq t^{\alpha}:$
Let $l=0,1,2. \quad \alpha <2. \quad \beta<1.$ Assume the initial data is in $H^2$ and moreover, $\|\langle A \rangle^{l+0}\psi(0)\|_2<\infty .$
Then,

\begin{align}
& \left\langle\left(\frac{A}{t^{\beta}}\right)^{l+1} \ln \frac{\langle A\rangle}{t^{\beta}} F_{A, C}\left(\frac{\left|2 p^{2} t-A\right|}{t^{\alpha}} \geqslant 1\right)\right\rangle_{T} \\
& \lesssim\langle\cdots\rangle_{1}+o(1) .
\end{align}
\end{theorem}
\begin{remark}
The condition on the initial data to also be in a domain of powers of $A$, can be removed if we restrict the estimate by an extra factor $F(|A|/t < M),$ $M$ arbitrary.    
\end{remark}
\begin{proof}
In this case, we first consider the following PROB:

\begin{equation*}
B_{M} \equiv F_{A}\left(\frac{A}{tM} \leq \frac{1}{2}\right) F_{A, C} \left( \frac{2 p^{2} t-A}{t^{\alpha}} \geqslant 1\right)+c . c \tag{B.1}
\end{equation*}

Note that

\begin{equation}
F_{A, C}=e^{+i \Delta t} F_{A}\left(-\frac{A}{t^{\alpha}} \geqslant 1\right) e^{-i \Delta t} . \tag{B.2}
\end{equation}

The Heisenberg Derivative w.r.t. $-\Delta$ is negative to leading order:

\begin{equation*}
D_{H_{0}} F_{A}=\tilde{F}_{A}^{\prime}\left(\frac{2 p^{2}}{Mt}-\frac{A}{Mt \cdot t}\right) \tag{B.3}
\end{equation*}

This is an identity, since we use an analytic form of $F_{A}$.
On the support of $F_{A, C}$ this derivative is
$$
\leqslant \tilde{F}_{A}^{\prime} \frac{t^{\alpha}}{M t} F_{A,C}+ c . c .
$$
to leading order.
The Heisenberg derivative of $F_{A, C}$ is also negative:

\begin{equation*}
D_{H_{0}} F_{A, C}=\partial_{t} F_{A,C}=-\frac{\alpha}{t} F_{A,C}^{\prime} . 
\end{equation*}

The commutator with the potential is also fast decaying, since on the boundary $\tilde{F}_{A}^{\prime}$, $A \sim M t$, and on boundary of $\tilde{F}_{A, C}^{\prime}$, we are on the Propagation Set, and $|x| \sim 2 p t.$
Now, we note that far away from the PS (prop.  set)
$$
\frac{2 p^{2}}{Mt } \gg \frac{A}{M^2t},
$$
showing that away from the $P S$, we get a good bound on $p^{2}$.
We want to improve this to $p^{4}$, by finding the decay on $M$ of the higher order terms. One such term comes from the potential. With sufficient decay in $|x|$ large, these terms are higher order.
The main contribution comes from Symmetrization.

The leading order commutator is negative modulo correction terms coming from symmetrizing.
We use, for $A, B$ self adjoint operators  and $A$ non-negative, with a common domain, invariant under the application of $A,B :$

\begin{align}\label{B5}
A B+B A=2 A^{1 / 2} B A^{1 / 2}+A^{1 / 2}\left[A^{1 / 2}, B\right] \\
+\left[B, A^{1 / 2}\right] A^{1 / 2}=2 A^{1 / 2} B A^{1 / 2}+\left[A^{1 / 2},\left[A^{1 / 2}, B\right]\right] .\tag{B.5} 
\end{align}

Using equation (B.5), we get

\begin{align*}
& \tilde{F}_{A}^{\prime}\left(2 p^{2}t-A\right) F_{A, C}+F_{A, C}\left(2 p^{2}t-A\right) \tilde{F}_{A}^{\prime} \\
& =2 \sqrt{\tilde{F}_{A}^{\prime}}\left[\left(2 p^{2}t-A\right) F_{A, C}\right] \sqrt{\tilde{F}_{A}^{\prime}}+ \\
& {\left[\sqrt{\tilde{F}_{A}^{\prime}},\left[\sqrt{\tilde{F}_{A}^{\prime}}, G_{A, C}\left(2 p^{2} t-A\right)\right]\right] \text { (B.6a) }} \\
& G_{A, C}=\left(2 p^{2} t-A\right) \cdot F_{A, C} \quad(B .6 b) 
\end{align*}

By the commutator expansion formula, we have

\begin{align}\label{B7}
& {\left[f\left(\frac{A}{t}\right), g\left(\frac{2 p^{2} t-A}{t^{\alpha}}\geq 1\right) \cdot\left(2 p^{2} t-A\right)\right]}  \\
& =\frac{1}{t} f^{\prime}\left(\frac{A}{t}\right)\left[A, g \cdot\left(2 p^{2} t-A\right)\right]+ \text { h.o.t. } \\
& =\frac{1}{t} f^{\prime}\left(\frac{A}{t}\right)\left\{g \cdot 4 p^{2} t+g^{\prime} \frac{4 p^{2} t}{t^{\alpha}} \cdot\left(2 p^{2} t-A\right)\right\}+\text { h.o.t. } \\
& \approx \frac{1}{t} f^{\prime}\left(\frac{A}{t}\right)\left\{g\left(4 p^{2} t-2 A\right)+g \cdot 2A+g^{\prime} \cdot\left\{t^{\alpha}-2 A\right\}\right\}+ \text { h.o.t. }
\end{align}

The higher order terms come from the next orders in the Commutator Expansion Lemma. So, in particular, the higher order term in the above equation takes one more derivative of $f$, which gives a factor of $1/t$, times the second commutator
$$
\left[A,\left[A, g \cdot\left(2 p^{2} t-A\right)\right ]\right ].
$$
This commutator with $A$ can not produce a higher order term, since commuting with $p^2$ gives back a constant times $p^2.$
Commuting with $g$ gives an extra factor of $t/t^{\alpha}$, which is, however, smaller than the factor $1/t$ we get from the second derivative of $f.$ 
The last line is an approximation as we replaced $2p^2-A$ by $t^{\alpha}$ on the support of $g'$ and $g''.$
The RHS of \eqref{B7} has a leading order term

\begin{equation*}
\frac{1}{t} f^{\prime}\left(\frac{A}{t}\right)\left(4 p^{2} t-2 A\right) g\left(\frac{-A+2 p^{2} t}{t^{\alpha}} \geqslant 1\right) . \tag{B.8}
\end{equation*}

The other two terms are of order 1 or lower in $t$, since $\frac{A}{t} f^{\prime}\left(\frac{A}{t}\right)$ is bounded.

Commuting

\begin{align*}
& {\left[\sqrt{F_{A}^{\prime}}, f^{\prime}\left(\frac{A}{t}\right)\left(2 p^{2} t-A\right) g\left(\frac{2 p^{2} t-A}{t^{\alpha}} \geqslant 1\right)\right]} \\
& =\frac{1}{t} f^{\prime}\left(\frac{A}{t}\right)^{2}\left[\left(4 p^{2} t\right) g+\frac{4 p^{2} t}{t^{\alpha}} g^{\prime}\left(2 p^{2} t-A\right)\right]+h .o . t . \tag{B.9}
\end{align*}

Here, $f\left(\frac{A}{t}\right) \equiv \sqrt{\tilde{F}_{A}^{\prime}} . \quad g\left(\frac{2 p^{2} t-A}{t^{\alpha}}\right) \equiv F_{A, C}$.
Writing $4 p^{2} t$ as $4 p^{2} t-2 A+2 A$, then the $g^{\prime}$ factor in B.9 is bounded by
$$
2 g^{\prime}+2 A g^{\prime} / t^{\alpha} .
$$

So, the leading term is
$$
\frac{1}{t} f^{\prime}\left(\frac{A}{t}\right)^{2} 4 p^{2} t g .
$$

The leading order remainder term, from \eqref{B7} and B.3 is therefore:

\begin{align*}
& \frac{1}{M^2 t^{2}} f^{\prime}(A / t)^{2} 4 p^{2} t g\left(\frac{2 p^{2} t-A}{t^{\alpha}} \geqslant 1\right)= \\
& \frac{1}{M t^{2}} \cdot \frac{1}{M^{2} t^{2}} f^{\prime}\left(\frac{A}{M t}=1\right)^{2} 4 p^{2} t g . \tag{B. 10}
\end{align*}

The leading order term on the RHS is

\begin{equation*}
\frac{1}{t M} \cdot \sqrt{\tilde{F}_{A}^{\prime} \,\left(\frac{A}{Mt}=1\right)}\left(2 p^{2} t-A\right) F_{A, C} \sqrt{\tilde{F}_{A}^{\prime}\,\left(\frac{A}{Mt }=1\right) }. \tag{B.11}
\end{equation*}

Since we proved before, equation \eqref{p-bd}, that $\left\langle p^{2} F\left(|p|>t^{1 / 5}\right)\right\rangle$ is bounded in time, it follows that the leading remainder term, B.10, is bounded by
$$
\left\langle\frac{1}{M^{3}} \frac{1}{t^{3}} f^{2} g t^{2/5}\right\rangle \in L^{1}(d t).
$$

In particular, we have the bound
$$
\int^{\top}\langle B. 10\rangle d s \leqslant \int^{\top}\left\langle\frac{1}{M^{3}} f^{\prime 2}\left(\frac{A}{M s} \sim 1\right) g\right\rangle_s \frac{d s}{s^{13/5}} .
$$

Multiply by $M, M^{2}$ or $M^{3}$, where
$$
M=M_{0} 2^{n} \quad n=0,1,2, \ldots
$$
and sum over $n,$
we get the following bounds:
$$
\begin{aligned}
& \int_{1}^{T} \frac{d t}{t}\left(\frac{A}{t}\right)^{l}\left\langle\tilde{F}\left(\frac{A}{t} \geqslant M_{0}\right)\left(2 p^{2}-\frac{A}{t}\right) F_{A, C} \tilde{F}_{A}\left(\frac{A}{t} \geqslant M_{0}\right)\right\rangle_{t}+ \\
& +\left\langle\left(\frac{A}{t}\right)^{l+1} \ln \left\langle\frac{A}{t}\right\rangle F_{A, C}+c. c.\right\rangle_{T}-\left\langle\left(\frac{A}{t}\right)^{l} \ln \langle A\rangle F_{A, C}\right\rangle_{1} \\
& \leq\left\langle\int_{1}^{T}\left\langle \bar F_{A}\left(\frac{A}{M_{0} s} \geqslant 1\right) F_{A, C}\right\rangle \frac{d s}{s^{13 / 5}}+\text { h.o.t. } \lesssim O(1) .\right. \\
& l=1,2 .
\end{aligned}
$$

\begin{remark}

The leading error term has decay in time that is more than needed for convergence. Therefore, we can redo the estimate, using
$$
B_{A, \beta} \equiv F_{A}\left(\frac{A}{Mt^{\beta}} \leqslant \frac{1}{2}\right) F_{A, c}\left(\frac{2 p^{2}-\beta A}{t^{\alpha}} \geqslant 1\right)+c \cdot c .
$$

The estimates we obtain show that on support $F_{A, C},  \left(\frac{A}{t}\right)^{2}$ and $\left(\frac{A}{t^{\alpha}}\right)^{2}$ are uniformly bounded, respectively.

It also follows that
$$
\int_{1}^{T}\left(\frac{A}{t^{\beta}}\right)^{2}\left\langle\tilde{F}_{A}\left(\frac{A}{t^{\beta}} \geqslant M_{0}\right) 2 p^{2} F_{A, C} \tilde{F}_{A}\right\rangle_{t} \frac{d t}{t^{\beta}}<\infty
$$
\end{remark}

Next, we note that the region $A-2 p^{2} t \geqslant t^{\alpha}$, which is also away from the PS, can be treated similarly:
In this case we use
$$
B_{M}^{-} \equiv F_{A}\left(\frac{A}{t M} \geqslant \frac{1}{2}\right) F_{A,C}\left(\frac{A-2 p^{2} t}{t^{\alpha}} \geqslant 1\right)+\text { c.c. }
$$
and notice that $B_{M}^{-} \geqslant 0+$ h.o.t. and the leading order Heisenberg derivative of $B_{M}$ is negative.
$$
\begin{aligned}
& D_{H_{0}} F_{A, C}=-\frac{\alpha}{t} F_{A, C}^{\prime} \leqslant 0 \\
& D_{H_{0}} F_{A}=\tilde{F}_{A}^{\prime}\left[p^{2}(Mt )^{-1}-\frac{A}{Mt^{2}}\right],
\end{aligned}
$$
which is also negative on the support of $F_{A, C}$. Similarly with $A / t M$ replaced by $A / t^{\beta} M$, $\beta<1$,but sufficient to control the decay of the potential terms.
The potential commutator is shown, as before, to have decay in the powers of $M$ and $t$.

The symmetrization works as before, as these operators are the same as before up to the sign change.
We conclude in particular that
$$
\begin{gathered}
\int_{1}^{T} \frac{d t}{t}\left\langle\left(\frac{A}{t^{\beta}}\right)^{l} \cdot \tilde{F}_{A}\left(\frac{A}{t^{\beta}} \geqslant M_{0}\right) p^{2} F_{A, C}\left(\frac{\left|2 p^{2} t-A\right|}{t^{\alpha}} \geqslant 1\right) \tilde{F}_{A}\left(\frac{A}{t^{\beta}} \geqslant M_{0}\right)\right\rangle_{t} \\
\leqslant\langle\cdots\rangle_{t=1}+O(1) . \quad l=0,1,2 .
\end{gathered}
$$

Furthermore, we have

\begin{align}
& \left\langle\left(\frac{A}{t^{\beta}}\right)^{l+1} \ln \frac{\langle A\rangle}{t^{\beta}} F_{A, C}\left(\frac{\left|2 p^{2} t-A\right|}{t^{\alpha}} \geqslant 1\right)\right\rangle_{T} \\
& \lesssim\langle\cdots\rangle_{1}+o(1) .
\end{align}

\end{proof}

\section{$H^{2}$ - Regularity on the Outgoing Waves}

\begin{theorem}

Let $\psi(t)$ be a solution of the SE as considered above.
Assume the initial condition is well localized in space and in $H^{2}\left(R^{3}\right)$.

Then,\newline
a) If $\sup _{t}\langle\phi(t), H(t) \phi(t)\rangle \lesssim O(1)$
where $\phi(t)=F\left(\frac{A}{t^{\alpha}} \geqslant 1\right) \psi(t)$,
then

\begin{equation}
\int_{1}^{T}\left\langle\Delta \psi(t),\left[\left\langle\frac{A}{t^{\alpha}}\right\rangle \ln \left\langle A / t^{\alpha}\right\rangle\right]^{-1} \Delta \psi(t)\right\rangle \frac{d t}{t^{\alpha}} \leq O(1) \tag{C.2}
\end{equation}

b) If for suitably localized $G(A)$ \begin{equation}
\sup_t \langle \nabla\phi(t),G(A)\langle\frac{A}{t^{\alpha}}\rangle\ln{\langle\frac{A}{t^{\alpha}}\rangle}  \nabla \phi(t)\rangle\leq O(1) \tag{C.3a}\end{equation}
then \begin{equation}
\int_{1}^{T} \frac{d t}{t}\left\langle\Delta \psi(t), G(A) F\left(\frac{A}{t^{\alpha}} \geqslant 1\right) \Delta \psi(t)\right\rangle \leq O(1). \tag{C.3b}\end{equation}  
\end{theorem}

\begin{proof}

\begin{align*}
\partial_{t}\langle\psi(t), & \left.F_{A, M}\left(\frac{A}{M t^{\alpha}} \geq 1\right) H(t) F_{A, M}\left(\frac{A}{M t^{\alpha}} \geq 1\right) \psi(t)\right\rangle \\
= & \left\langle\frac{1}{M t^{\alpha}} \tilde{F}_{A, M}^{\prime}\left[2 p^{2}-\frac{\alpha A}{t}\right] H(t) F_{A, M}+c . c .\right\rangle \\
& +\left\langle\psi(t), F_{A, M} \frac{\partial V}{\partial t} F_{A, M} \psi(t)\right\rangle  \tag{C.4}\\
& +\left\langle\psi(t),\left(i\left[V, F_{A, M}\right] H(t) F_{A, M}+c . c .\right) \psi(t)\right\rangle \equiv I_{1}+I_{2}+I_{3}
\end{align*}

We proved in the previous section estimates away from the PS.

So, to deal with the first term, we write

\begin{align*}
& {\left[2 p^{2}-\frac{\alpha A}{t}\right]=\left[2 p^{2}-\frac{\alpha A}{t}\right] F_{A, C}^l\left(\frac{2 p^{2} t-A}{t^{\alpha}} \geqslant 1\right) } \\
+ & {\left[2 p^{2}-\frac{\alpha A}{t}\right] F_{A, C}^L\left(\frac{A-2 p^{2} t}{t^{\alpha}} \geqslant 1\right) } \\
+ & {\left[2 p^{2}-\frac{\alpha A}{t}\right] F_{A, C}^0\left(\frac{\left|A-2 p^{2} t\right|}{t^{\alpha}} \leq 1\right) } \tag{C.5}
\end{align*}

The first and third terms in equation  C.5 are positive (for $2p^2 t\geq t^{\alpha}$).
The second term is supported where $\frac{A}{t} \geqslant 2 p^{2}+t^{\alpha-1}$ and multiplied by $\tilde{F}_{A, M}^{'}$ where $\frac{A}{t^{\alpha}} \sim M$.

Therefore, $M t^{\alpha-1} \geqslant 2 p^{2}+t^{\alpha-1}$.

Consequently,

\begin{align*}
& \left\langle\tilde{F}_{A, M}^{\prime}\left[2 p^{2}-\frac{\alpha A}{t}\right] F_{A, C}^L\left(\frac{A-2 p^{2} t}{t^{\alpha}} \geqslant 1\right) H(t) F_{A, M}\right) \\
\lesssim & \left\langle\tilde{F}_{A, M}^{\prime} M^{2} t^{2 \alpha-2} F_{A, C}^L F_{A, M}\right\rangle  \tag{C.6}\\
& \text { since } F_{2}\left(p^{2} \leq M t^{\alpha-1}\right) H(t) \lesssim F_{2} M t^{\alpha-1}+F_{1} V
\end{align*}

and
$$
\begin{aligned}
\tilde{F}_{A, M}^{\prime} \cdot F_{1} V \sim & \tilde{F}_{A, M}^{\prime}\langle A\rangle^{-2} F_{1} A^{2} V \sim \\
& \left\langle M t^{\alpha}\right\rangle^{-2} M t^{\alpha-1}\left( |x|^{2} V\right) \sim M^{-1} t^{-1-\alpha}.
\end{aligned}
$$

So,

\begin{align*}
I_{1} &\geqslant\langle (Mt^{\alpha})^{-1}\tilde F'_{A,M}2p^4 F_{A,C}^lF_{A,M}\rangle-\left\langle\left(M t^{\alpha}\right)^{-1} \tilde F_{A, M}^{'} M^{2} t^{2 \alpha-2} F_{A, C}^l\left(\frac{A-2 p^{2} t}{t^{\alpha}} \geqslant 1\right) F_{A, M}\right\rangle \\
I_{2} & =\left\langle\psi(t), F_{A, M} \frac{\partial V}{\partial t} F_{A, M} \psi(t)\right\rangle  \tag{C.7}\\
& \approx\left\langle\psi(t), F_{A, M}\langle A\rangle^{-2}\left(A^{2}+1\right)\langle x\rangle^{-m}\left(A^{2}+1\right)\langle A\rangle^{-2} F_{A, M} \psi(t)\right\rangle  \tag{C.8}\\
& \approx\left\langle\Delta \psi(t),\langle x\rangle^{-\sigma} \tilde{F}_{A,M} t^{-4 \alpha}\langle x\rangle^{-m+4+2 \sigma} \tilde{F}_{A, M}\langle x\rangle^{-\sigma} \Delta \psi(t)\right\rangle M^{-4} \\
& \approx O\left(t^{-4 \alpha}\right)\|\Delta \psi(0)\|_{2}^{2} M^{-4},
\end{align*}

provided $m \geqslant 4+2 \sigma ; \quad \psi(0) \in H^{2}\left(\mathbb{R}^{3}\right)$.
\begin{remark}
The above estimate with $m\geqslant 4+2 \sigma$ is not necessary, and $\geqslant 2+2 \sigma$ is sufficient.
In the above we demonstrate the maximal decay possible at the $H^2$ level.
\end{remark}

The control of  $I_3$ is similar:
In this case, due to the $H(t)$ factor, we can only multiply by $\langle A\rangle^{2}$ and divide.
By restricting the domain to the PS or where
$A \geqslant 2 p^{2} t$, two powers of $\langle A\rangle ^{-1}$ give a decay of order $t^{-2}$, so, such regions are bounded by
$$
\|\psi(0)\|_{H^{2}}^{2} t^{-2}\left(M t^{2}\right)^{-1} .
$$

In the region $A \leq 2 p^{2} t$, the $\langle A\rangle^{-2}$ gives
$$
\|\psi(0)\|_{H^{2}}^{2}\left(M t^{\alpha}\right)^{-2}
$$

We also used the Commutator Expansion Lemma for
$$
\left[V, F_{A, M}\right] \sim x \cdot \nabla V F_{A, M}^{\prime}\left(M t^{\alpha}\right)^{-1}+\text { h.o.t. }
$$

So,

\begin{equation*}
I_{3} \lesssim\left[\left(M t^{\alpha}\right)^{-3}+t^{-2}\left(M t^{\alpha}\right)^{-1}\right]\|\psi(0)\|_{H^{2}}^{2} \tag{C.9}
\end{equation*}

Combining the equations C.4-C.9
we get

\begin{align*}
& \left\langle F_{A, M} H(t) F_{A, M}\right\rangle_{t=T}-\left\langle F_{A, M} H(t) F_{A, M}\right\rangle_{t=1} \\
& =\int_{1}^{T} \frac{d t}{t^{\alpha}}\left\langle\frac{1}{M} \tilde{F}_{A, M}^{\prime} 2 p^{2} F_{A, C}^{+} p^{2} \tilde{F}_{A, M}^{\prime}\right\rangle+h.o.t. \\
& +\int_{1}^{T} t^{2\alpha} \frac{d t}{t^{2} t^{\alpha}} \frac{M^{2}}{M}\left\langle\tilde{F}_{A, M}^{\prime} F_{2}\left(p^{2} \leq M t^{\alpha-1}\right) F_{A, C}^{-} F_{A, M}+c . c .\right\rangle \\
& +\int_{1}^{T} d t\left\{O\left(t^{-4 \alpha}M^{-4}\right)+\left(M t^{\alpha}\right)^{-3}+O\left(t^{-2}\right)\left(M t^{\alpha}\right)^{-1}\right\}^{2} \|\Delta \psi(0)\|_{2}^{2}. \tag{C.10}
\end{align*}

Here,
$$
F_{A, C}^{+}=F_{A, C}\left(\frac{\left|2 p^{2} t-A\right|}{t^{\alpha}} \leq 1\right)+F_{A, C}\left(\frac{2 p^{2} t-A}{t^{\alpha}} \geqslant 1\right)
$$

We use $\tilde{F}_{A, M}^{\prime}$ to denote a generic function of the form $\widetilde{F}_{A, M}\left(\frac{A}{M t^{\alpha}}=1\right)$.

Next, we will use that
$$
\sum_{n=0}^{\infty}\left(M_{0} 2^{n}\right)^{l} F_{A, M_{0} 2^{n}} \sim\left(\frac{A}{t^{\alpha}}\right)^{l} \ln \left(\langle A\rangle / t^{\alpha}\right) F_{A}\left(\frac{A}{t^{\alpha}} \geqslant M_{0}\right)
$$
and
$$
\sum_{n=0}^{\infty}\left( M_02^{n}\right)^{l} / n F_{A, M_0 2^{n}} \sim\left(\frac{A}{t^{\alpha}}\right)^{l} F_{A}\left(\frac{A}{t^{\alpha}} \geq M_0\right)
$$

We sum over $M=M_{0} 2^{n}, n=0,1,2, \ldots$ both sides of eq. C.10, and sum over $n$ up to $M\sim cT^{1-\alpha}.$

{\bf The LHS becomes}
\begin{align*}
&\left\langle  F_{A}\left(cT^{1-\alpha}\geq\frac{A}{t^{\alpha}} \geq M_{0}\right) H(t) F_{A}\left(cT^{1-\alpha}\geq\frac{A}{t^{\alpha}} \geq M_{0}\right) \ln T\right\rangle_{t}\\
&-\left\langle  F_{A}\left(cT\geq A \geqslant M_{0}\right) \ln T \, H(1) F_{A}\left(cT\geq A\geq M_{0}\right)\right\rangle.
\end{align*}

{\bf The RHS}

The first term
$$
 \int_{1}^{T} d t\left\langle \langle A\rangle^{-1} F_{A}\left(cT^{1-\alpha}\geq \frac{A}{t^{\alpha}} \geqslant M_{0}\right) p^{2} F_{A,C}^{+} p^{2} F_{A}\left(cT^{1-\alpha}\geq\frac{A}{t^{\alpha}} \geqslant M_{0}\right)\right\rangle\sim
 $$
$$
\frac{1}{cT}\int_1^T \left \langle F_A p^2 F^+_{A,C} p^2 F_A\right \rangle d t.
$$

The second term
$$
\int_{1}^{T} \frac{d t}{t^{2}}\left\langle A F_{A}\left(cT^{1-\alpha}\geq\frac{A}{t^{\alpha}} \geqslant M_{0}\right) F_2 F_{A ,C}^{-} F_{A}\left(cT^{1-\alpha}\geq \frac{A}{t^{\alpha}} \geqslant 1\right)\right\rangle
$$

The third term
$$
C+\int_{1}^{T} \frac{d t}{t^{2+\alpha}}\|\psi(0)\|_{H^{2}}^{2}
$$

The LHS is uniformly bounded by $\ln T$ since $A\leq cT.$
We conclude that a sequence of times going to infinity exists such that the $H^2$ norm remains uniformly bounded on the phase space support of the localization of $A.$ This bound is then extended to all times, using as PROB $p^2G(A,t)p^2$ with $G$ localizing at the relevant domain in phase space. \end{proof}

\section*{High Frequency on the Propagation Set}

\begin{theorem}
We have the following estimate on the PS:

\begin{equation}
\left<H(T)\, F_c\,H(T)\right>_{T}\lesssim O(1)\,\|\psi_0\|_{L^2}\,\|\psi_0\|_{H^2}.
\end{equation}

\end{theorem}

\begin{proof}
To control the $H^2$ norm of the solution on the PS
we consider the following PROB:

\begin{equation}
B_{HH}
=\left<\psi(t),H(t)\,F_c\!\left(\frac{|x-2pt|}{R}\leq 1\right)H(t)\psi(t)\right>.
\end{equation}

Here $R$ is a large number that may depend on $t$.

Since
\begin{equation}
D_H F_c
=
\partial_t F_c
+
i[-\Delta,F_c]
=
-\frac{|x-2pt|}{R^2}\,\dot R\,F_c',
\end{equation}
it is zero if $\dot R =0$.

Hence, for $\dot R=0$, we get
\begin{equation}
\partial_t B_{HH}
=
\left<
\psi(t),
\left(
\frac{\partial V}{\partial t}
F_c\,H(t)
+
H(t)F_c\frac{\partial V}{\partial t}
\right)
\psi(t)
\right>
\end{equation}

\[
\qquad
+
\left<
\psi(t),
H(t)\,i[V,F_c]\,H(t)\psi(t)
\right>.
\]

For arbitrarily large $T$, we want to estimate
$B_{HH}(T)$ in the PS; for this, we need to take
$R$ at least of order $t^\alpha$, for some $\alpha<1$.

We have two types of terms to control:

\begin{equation}
\text{a)}\qquad
H(t)\,i[V,F_c]\,H(t)
=
p^2\,i[V,F_c]\,p^2
+
O(V)\,i[V,F_c]
\end{equation}

\begin{equation}
\text{b)}\qquad
\frac{\partial V}{\partial t}\,
F_c\,H(t)
\sim
\langle x\rangle^{-\sigma}
F_c\,p^2
+
O(V)\,
\frac{\partial V}{\partial t}\,
F_c.
\end{equation}

We concentrate on the type (a) that is more
difficult. The second part (b) is treated similarly.

The key estimate we use is the gain in
regularity and smoothness of the product
of $V$ and $\widetilde F_c$:

\begin{equation}
V(x,t)\, F_c=V(x,t)\,
\langle x\rangle^\sigma\langle x\rangle^{-\sigma}
e^{i\Delta t}F_2(|p|\geq M)F_1\!\left(\frac{|x|}{R}\leq 1\right)e^{-i\Delta t}F_2(|p|\ge M)
\end{equation}

Therefore, we estimate
\[\langle x\rangle^{-\sigma}e^{i\Delta t}F_2F_1\!\left(\frac{|x|}{R}\leq 1\right).\]

A first observation is

\begin{equation}
\langle x\rangle^{-\sigma}e^{i\Delta t}F_2
F_1\!\left(\frac{|x|}{R}\leq 1\right)=\langle x\rangle^{-\sigma}\langle x-2pt\rangle^{-2}e^{i\Delta t}\widetilde F_2R^{2}\widetilde F_1\!\left(\frac{|x|}{R}\leq 1\right)
\end{equation}

where

\begin{equation}
\langle x\rangle^2F_1\!\left(\frac{|x|}{R}\leq 1\right)=R^2\widetilde F_1\!\left(\frac{|x|}{R}\leq 1\right)
\end{equation}

and
\[\widetilde F_2 \approx F_2 +\mathcal{O}( \frac{1}{M}).\]

\begin{proposition}

For $\sigma>2$,

\begin{equation}
\langle x\rangle^{-\sigma}e^{i\Delta t}F_2F_1\!\left(\frac{|x|}{R}\leq 1\right)=O(\langle x\rangle^{-m})\langle pt\rangle^{-2}R^2O(1)\,\widetilde F_1,\qquad
m=\sigma-2.
\end{equation}
\end{proposition}
\begin{proof}
\begin{equation}
\langle x\rangle^{-2}\langle pt\rangle^{-2}\langle pt\rangle^{2}F_2\langle x-2pt\rangle^{-2}F_2=
\end{equation}

\[
\langle x\rangle^{-2}\langle pt\rangle^{-2}\Big(1+p^2t^2-t\,2x\!\cdot\! p-2pt\!\cdot\! x+x^2-x^2\Big)F_2\langle x-2pt\rangle^{-2}F_2=
\]

\[
\langle x\rangle^{-\sigma}\langle pt\rangle^{-2}\widetilde F_2+\langle x\rangle^{-\sigma}\langle pt\rangle^{-2}[-x^2-t\,2p\!\cdot\!x-t\,2x\!\cdot\!p]\widetilde F_2\langle x-2pt\rangle^{-2}F_2
\]

\begin{equation}
\langle x\rangle^{-2}\langle pt\rangle^{-2}(x^2-2x\!\cdot\!pt)=
\end{equation}

\[
=\frac{x^2}{2\langle x\rangle^2}\langle pt\rangle^{-2}-\langle x\rangle^{-2}\partial_p^2\langle pt\rangle^{-2}-\langle x\rangle^{-2}2x\langle pt\rangle^{-2}\cdot pt
\]

\[
\langle x\rangle^{-2}
\langle pt\rangle^{-2}
2pt\!\cdot\!x=
\langle x\rangle^{-2}x
\langle pt\rangle^{-2}2pt
-\langle x\rangle^{-1}O(1)\langle pt\rangle^{-2}2t
\]
$\mathcal{O}(1)$ stands for a bounded operator coming from commutators.
Multiplying from the right by $\widetilde F_2$, we get

\begin{equation}
\langle x\rangle^{-2}\langle pt\rangle^{-2}(x^2-2x\!\cdot\!pt+2pt\!\cdot\!x)\widetilde F_2=
\langle pt\rangle^{-2}(1+O(1))\widetilde F_2
\end{equation}

\[
\qquad+\langle x\rangle^{-2}\langle pt\rangle^{-2}(1+O(1))
\frac{p^2t^2}{1+p^2t^2}\frac{1}{M^2}\widetilde F_2
\]

\[
\qquad+\langle x\rangle^{-2}\langle pt\rangle^{-2}2pt\left(\frac{\widetilde F'}{M}\right).
\]

Hence,

\begin{equation}
\langle x\rangle^{-2}
\widetilde F_2
\langle x+2pt\rangle^{-2}
\widetilde F_2
=
(1+O(1))
\langle pt\rangle^{-2}
\widetilde F_2
+\langle x\rangle^{-2}\langle pt\rangle^{-1}\left(\frac{\widetilde F'_2}{M}\right)\langle x-2pt\rangle^{-2}F_2
\end{equation}

Therefore,

\begin{equation}
\langle x\rangle^{-2}(1+O(1)\langle pt\rangle^{-1})\widetilde F_2\langle x-2pt\rangle^{-2}\widetilde F_2=(1+O(1))\langle pt\rangle^{-2}\widetilde F_2
\end{equation}

Therefore,
\[\left(M\gg1,\quad\langle x\rangle^{-2}O(1/M)\langle pt\rangle^{-1}\ll1\right)
\]

\begin{equation}
\langle x\rangle^{-2}\widetilde F_2\langle x-2pt\rangle^{-2}\widetilde F_2\approx\langle pt\rangle^{-2}\widetilde F_2
\end{equation}

 We also have the following

\begin{align}
&\|\langle x\rangle^{-2}\widetilde F_2\langle x-2pt\rangle^{-2}\widetilde F_2\phi\|=\\
&\|\langle x\rangle^{-2}\widetilde F_2\langle x-2pt\rangle^{-2}\langle x\rangle^{2}\langle x\rangle^{-2}\widetilde F_2\phi\|\lesssim\|\langle x\rangle^{-2}\widetilde F_2\phi\|
\end{align}
We also used that
\[
i[x,F_2]=-\frac{F_2'}{M},
\]
and therefore
\begin{equation}
[\langle x\rangle^{-\sigma},F_2]
\sim\langle x\rangle^{-\sigma-1}\left (\frac{F_2'}{M}+O\!\left(\frac{F_2''}{M^2}\right)\right).
\end{equation}

For $|p|\le \varepsilon$ we have:

\[
\langle x\rangle^{-\sigma}\langle pt\rangle^{-2}\langle pt\rangle^{2}\widetilde F_2\langle x-2pt\rangle^{-2}\widetilde F_2
\]

\[
\lesssim\langle pt\rangle^{-2}(1+\varepsilon^2 t^2)t^{-3/2}\widetilde F_2e^{-i\Delta t}
\]


\[
\lesssim C\,\varepsilon^2t^2t^{-3/2}\langle pt\rangle^{-2}=C\,\varepsilon^{2}t^{1/2}\langle pt\rangle^{-2}.
\]
as operators on
\[
L^2 \to L^2(\mathbb{R}^3)
\]
In the above we also used the pointwise time decay estimate $\|\langle x\rangle^{-\sigma} e^{i\Delta t}\langle x\rangle^{-\sigma}\|_{2\rightarrow 2} \leq ct^{-3/2}$ in three dimensions. This follows from the standard $L^{\infty}$ decay, provided $\sigma > 3/2,$ and with the $\tilde F_2$ for $\sigma \geq n/2.$
In higher dimensions, we get similar estimates,
also by using $\langle x\rangle^{-\sigma}$ instead of
$\langle x\rangle^{-2}$, $\sigma>2$.

In $4$-dimensions, we get with
$\sigma=2+\delta$ on the left, the bound

\[
C\,\varepsilon^2t^2t^{-2+\delta}\langle pt\rangle^{-2},
\]

since
\[
\langle x\rangle^{-2-\delta}\phi\in L^p\qquad\text{for all } p>1, \phi \in L^2.
\]
\end{proof}
\begin{corollary}

\begin{equation}
p^2\,i[V,F_2F_c F_2]\,p^2
=
p^2
\widetilde V
\langle x\rangle^{-\sigma}
\langle x\rangle^{-2} F_2 F_c  F_2 p^2+c.c.=
\end{equation}

\[
=p^2\langle x\rangle^{-\sigma}\widetilde V\langle pt\rangle^{-2}(1+O(1))\widetilde F_2p^2R^2+c.c.
\]

\[
=p^2\langle x\rangle^{-\sigma}\widetilde V\langle t\rangle^{-2}O(1)\widetilde F_2R^2+c.c.
\]

Consequently,

\begin{equation}
\int_{T_1}^{T}\left<p^2\,i[V,F_2F_c\widetilde F_2]\,p^2\right>dt\lesssim\langle T\rangle^{-1}\|\widetilde V\|_{\infty}\|\psi(0)\|_{H^2}\|\psi(0)\|_{L^2}R^2.
\end{equation}
Since
\[
R\sim T^\alpha,
\]
we need
\[
\alpha<\frac12.
\]

This gives control over the high frequency
part on the PS:
\[
F_2\,F_c\,\widetilde F_2.
\]
\end{corollary}

To cover the PS over the time interval $[0,T],$
we break the interval into sub-intervals; the first is
$(\delta \ll 1)$
\[
[\delta T,T]
\]
and we choose
\[
R=T^{2/5}.
\]

Then, we get

\begin{equation}
\left<H(T)\, F_c\,H(T)\right>_{T}-\left<H(T_1)\,F_c\,H(T_1)\right>_{T_1=T^{4/5}}=
\end{equation}

\[
=\int_{T_1}^{T}\partial_s\left<\psi(s),H(s)\,F_{c,R}\,H(s)\psi(s)\right>\lesssim T_1^{-1}R^2\|\psi\|_{H^2}\|\psi(0)\|_{L^2}
\]

\[
\lesssim O(1)\,\|\psi_0\|_{L^2}\,\|\psi_0\|_{H^2}.\]

Similarly,

\begin{equation}
\left<H(T_1)\, F_c \,H(T_1)\right>-\left<H(T_2)\,F_c\,H(T_2)\right>_{T_2=T_1^{4/5}}
\end{equation}

\[
\lesssim
O(1)\,\|\psi_0\|_{L^2}\,\|\psi_0\|_{H^2}.
\]

We have

\begin{equation}
T_n
=
T_{n-1}^{4/5}
=
(T_{n-2})^{\frac45\cdot\frac45}
=
\cdots
=
T_1^{(4/5)^{n-1}}
=
T^{(4/5)^n}
\end{equation}

Hence, for $T_n\sim1$, we need

\begin{equation}
\left(\frac45\right)^n\ln T
=
O(1),
\end{equation}

or

\begin{equation}
\ln\ln T
\sim
n(\ln 5/4).
\end{equation}

In order to remove the growth of $\ln\ln T$, it is sufficient to replace the factor $2/5$ by $2/5-0$ in the definition of $F_c.$
Then each interval contribution will be suppressed by a factor $ T_n^{-\epsilon}, $ so the sum over all contributions is finite, since $T_n$ converges to 1 super exponentially fast, and the number of terms is $c\ln\ln T.$
The resulting estimate of the norm $H^2$ completes the part of the phase-space left from the previous section.
\end{proof}

\section{ Final Comments}

There are problems where higher regularity cannot be achieved without proving Local decay estimates and other $L^p$ decay estimates.
The approach introduced in this work does not use such methods.

Some simple examples of applications include potentials that correspond to modulated soliton solutions of NLS type equations.
The standard example is the soliton part of the dynamics in the modulation equations that couples the changing Soliton to the Radiation.
In this case, the potential is essentially a smooth localized function with width and amplitude varying in an unknown way in time.

The general case of time periodic, quasi periodic and some almost periodic space-localized potentials is covered.

Also included Saturated Nonlinear Equations, where the nonlinear term is of the general form
$$
I(x,t)\frac{|\psi|^2}{1+V(x,t)|\psi|^2}
$$
with smooth localized function $I$, and smooth localized plus a constant function $V.$

More intricate is the case of semilinear equations. First, it should be observed that the Bootstrap step in the time interval $t-\epsilon,t]$ in the proof of Theorem 2.2 can be extended to the nonlinear case.  Observe that the nonlinear bootstrap estimate is sufficient, simply choosing the size of the time interval $\epsilon,$ small compared to the power of the size of the $H^2$ norm of the Initial Data. This shows that the solution is in $H^2$ locally in space and uniform in time. Localization of the interaction term for the next steps can be obtained from the assumption of spherical symmetry.

\section{ Appendix A}
\subsection{A dilation-analytic cutoff resolvent}

Let
\[
 p=-i\nabla,
 \qquad
 A=\frac12(x\cdot p+p\cdot x)
\]
be the generator of dilations on $L^2(\mathbb R^d)$. Thus,
\[
 [A,p^2]=2i p^2.
\]
Let $\chi_-\in C^\infty(\mathbb R;[0,1])$ be a smooth cutoff satisfying, for example,
\[
 \chi_-(a)=1\quad(a\leq -2),
 \qquad
 \chi_-(a)=0\quad(a\geq -1),
\]
and set
\[
 P^-(A):=\chi_-(A),
 \qquad
 G:=A P^-(A)=A\chi_-(A).
\]
Since $\chi_-$ is supported in the negative spectral region of $A$,
\[
 G\leq 0.
\]
Notice that no analyticity assumption on the scalar cutoff $\chi_-$ is required.

For $\tau>0$ and a complex deformation parameter $\beta$, define formally
\[
 T_\beta=G-2\tau e^{2i\beta}p^2.
\]
It is preferable to realize this operator by sectorial forms.

\begin{proposition}[Dilation-analytic inverse]
Let
\[
 \mathcal{Q}=\mathcal{D}((-G)^{1/2})\cap \mathcal{D}(p),
\]
and, for $\beta\in\mathbb C$ with
\[
 |\operatorname{Re}\beta|<\frac{\pi}{4},
\]
define the sesquilinear form
\[
 h_\beta[u,v]
 =
 \langle (-G)^{1/2}u,(-G)^{1/2}v\rangle
 +2\tau e^{2i\beta}\langle pu,pv\rangle,
 \qquad u,v\in\mathcal{Q}.
\]
Let $H_\beta$ be the $m$-sectorial operator associated with $h_\beta$, and put
\[
 T_\beta:=-H_\beta.
\]
Then:

\begin{enumerate}
\item The family $\beta\mapsto H_\beta$ is an analytic family of type~(B) in the strip
\[
 |\operatorname{Re}\beta|<\frac{\pi}{4}.
\]
Equivalently, $\beta\mapsto T_\beta$ is analytic in the same sense.

\item For every $\beta$ in this strip, the points $\pm i$ belong to the resolvent set of $H_\beta$.  Consequently,
\[
 1+T_\beta^2
\]
is invertible and
\[
 \boxed{
 (1+T_\beta^2)^{-1}
 =(H_\beta-i)^{-1}(H_\beta+i)^{-1}.
 }
\]

\item The map
\[
 \boxed{
 \beta\longmapsto (1+T_\beta^2)^{-1}
 }
\]
is a bounded-operator-valued analytic function throughout the strip
\[
 |\operatorname{Re}\beta|<\frac{\pi}{4}.
\]

\item In particular, for real $\beta$ with $|\beta|<\pi/4$,
\[
 \boxed{
 \|(1+T_\beta^2)^{-1}\|
 \leq \frac{1}{\cos(2|\beta|)}.
 }
\]
The bound is uniform on every closed subinterval
$|\beta|\leq\beta_0<\pi/4$.
\end{enumerate}
\end{proposition}

\begin{proof}
The form domain $\mathcal{Q}$ is independent of $\beta$, and for fixed $u,v\in \mathcal{Q}$ the map
\[
 \beta\longmapsto h_\beta[u,v]
\]
is entire.  Moreover, if $\beta=s+i t$, then
\[
 e^{2i\beta}=e^{-2t}e^{2i s}.
\]
Hence, whenever $|s|<\pi/4$, the second term in $h_\beta[u,u]$ lies in a sector of opening strictly smaller than $\pi$, centered on the positive real axis.  Since the first term is nonnegative real, the numerical range of $h_\beta$ is contained in a sector of semi-angle $2|s|<\pi/2$.  Standard closed-sectorial-form theory therefore gives an $m$-sectorial operator $H_\beta$, and Kato's analytic-form theorem shows that $H_\beta$ is an analytic family of type~(B).

For real $\beta$ with $0\leq\beta<\pi/4$, write
\[
 a=\|(-G)^{1/2}u\|^2\geq0,
 \qquad
 b=2\tau\|pu\|^2\geq0.
\]
Then
\[
 h_\beta[u,u]=a+e^{2i\beta}b,
\]
so the numerical range of $H_\beta$ is contained in
\[
 S_{2\beta}:=\{re^{i\theta}:r\geq0,\ 0\leq\theta\leq2\beta\}.
\]
Therefore
\[
 \operatorname{dist}(-i,S_{2\beta})=1,
 \qquad
 \operatorname{dist}(i,S_{2\beta})=\cos(2\beta).
\]
The resolvent estimate for $m$-sectorial operators gives
\[
 \|(H_\beta+i)^{-1}\|\leq1,
 \qquad
 \|(H_\beta-i)^{-1}\|\leq\frac1{\cos(2\beta)}.
\]
For negative real $\beta$, the sector is reflected across the real axis and the same estimate holds with $|\beta|$.
Thus, $T_{\beta}$ is the realization of the closed operator, in the form sense, of the formal expression $G-2\tau p^2e^{2i\beta}. $

Since $T_\beta=-H_\beta$,
\[
 1+T_\beta^2=1+H_\beta^2
 =(H_\beta-i)(H_\beta+i).
\]
on the natural domain of $H_\beta^2$.  Because the two resolvents are bounded functions of the same closed operator, they commute, and hence
\[
 (1+T_\beta^2)^{-1}
 =(H_\beta-i)^{-1}(H_\beta+i)^{-1}.
\]
This proves, for real $\beta$,
\[
 \|(1+T_\beta^2)^{-1}\|
 \leq
 \|(H_\beta-i)^{-1}\|\,
 \|(H_\beta+i)^{-1}\|
 \leq\frac1{\cos(2|\beta|)}.
\]

For complex $\beta$ with $|\operatorname{Re}\beta|<\pi/4$, the same sectorial argument shows that $\pm i$ remain outside the closed numerical-range sector.  Hence they remain in the resolvent set.  Analyticity of
\[
 \beta\mapsto(H_\beta\mp i)^{-1}
\]
follows from analytic perturbation theory for type~(B) families, and therefore their product
\[
 \beta\mapsto(1+T_\beta^2)^{-1}
\]
is analytic as a bounded-operator-valued function.
\end{proof}

\begin{remark}[Why the cutoff need not be analytic]
The scalar function $\chi_-$ is merely smooth.  This causes no obstruction to dilation analyticity because $P^-(A)=\chi_-(A)$ is a function of the dilation generator itself and is therefore fixed under dilation conjugation:
\[
 e^{i\theta A}P^-(A)e^{-i\theta A}=P^-(A).
\]
By contrast,
\[
 e^{i\theta A}p^2e^{-i\theta A}=e^{-2\theta}p^2
\]
for real dilation parameter $\theta$ (equivalently, after the customary complex reparametrization, the coefficient of $p^2$ becomes $e^{2i\beta}$).  Thus all dependence on the complex deformation parameter is carried by the scalar factor multiplying $p^2$.
\end{remark}

\begin{remark}[Advantage of the negative dilation cutoff]
The condition
\[
 AP^-(A)\leq0
\]
is the decisive simplification.  Both terms of
\[
 H_\beta=-AP^-(A)+2\tau e^{2i\beta}p^2
\]
then lie in the same sector.  No large positive-$A$ contribution occurs, and there is no need to control a transition term of size $R/M$.  In particular, the admissible deformation angle is determined by the sector condition
\[
 |\operatorname{Re}\beta|<\frac{\pi}{4},
\]
rather than by a parameter such as $R^{-1}$.
\end{remark}

{\bf Acknowledgment }

\thanks{
A.S. is supported in part by NSF grant DMS-2205931}

\thanks{soffer@math.rutgers.edu}

\bibliographystyle{plainnat}
\bibliography{bib}{}

\end{document}